\documentclass[11pt]{amsart}
\usepackage{amsmath}
\usepackage{amsfonts}
\usepackage{amssymb}
\usepackage{amsthm}
\usepackage{mathrsfs}       
\usepackage{graphicx}     
\usepackage[utf8]{inputenc}
\allowdisplaybreaks

\usepackage{xspace}
\usepackage[colorlinks=true]{hyperref}

\newcommand{\Sph}{\mathbb{S}}
\newcommand{\R}{\mathbb{R}}
\newcommand{\N}{\mathbb{N}}
\newcommand{\T}{\mathbb{T}}
\newcommand{\Z}{\mathbb{Z}}

\newcommand{\C}{\mathscr{C}}

\newcommand{\V}{\mathcal{V}}
\newcommand{\A}{\mathcal{A}}
\newcommand{\Y}{\mathcal{Y}}

\newcommand{\ud}{\,\mathrm{d}}

\renewcommand{\d}{\partial}

\renewcommand{\S}{\mathscr S}

\let \div \relax
\DeclareMathOperator{\div}{div}

\let \curl \relax
\DeclareMathOperator{\curl}{curl}

\DeclareMathOperator{\dist}{dist}

\newcommand{\ovl}[1]{\overline{#1}}
\newcommand{\udl}[1]{\underline{#1}}

\renewcommand{\O}{\mathcal O}

\newcommand{\trans}{\mathcal T}

\DeclareMathOperator{\supp}{supp}

\newtheorem{thm}{THEOREM}[section]
\newtheorem{remark}[thm]{REMARK}
\newtheorem{lemma}[thm]{LEMMA}
\newtheorem{definition}[thm]{DEFINITION}
\newtheorem{proposition}[thm]{PROPOSITION}

\newcounter{thmbiss}

\title{On the controllability of the 2-D Vlasov-Stokes system}

 \author[I.~Moyano]{Iv\'an Moyano}
\begin{address}{Iv\'an Moyano,
Centre de math\'ematiques Laurent Schwartz, UMR 7640, Ecole polytechnique, Palaiseau, France}
  \email{\href{mailto:ivan.moyano@math.polytechnique.fr}{ivan.moyano@math.polytechnique.fr}}
 \end{address}
\thanks{{\em Acknowledgment:}
The author would like to very much thank Daniel Han-Kwan (CNRS and CMLS, Ecole Polytechnique) for suggesting him this problem and for many fruitful discussions and advices.
}

\begin{document}
\maketitle

\begin{abstract}
In this paper we prove an exact controllability result for the Vlasov-Stokes system in the two-dimensional torus with small data by means of an internal control. We show that one can steer, in arbitrarily small time, any initial datum of class $\C^1$ satisfying a smallness condition in certain weighted spaces to any final state satisfying the same conditions. The proof of the main result is achieved thanks to the return method and a Leray-Schauder fixed-point argument. \\  

\end{abstract}

\section{Introduction}
We consider the Vlasov-Stokes system in the 2-dimensional torus $\T^2 := \R^2 / \Z^2$, which writes, for $T>0$ and $\omega \subset \T^2$, 
\begin{equation}
\left\{  \begin{array}{ll}
\partial_t f + v\cdot \nabla_x f + \lambda \div_v\left[ (U-v)f  \right]= 1_{\omega}(x)G, & (t,x,v) \in (0,T) \times \T^2 \times \R^2, \\
-\Delta_x U + \nabla_x p = j_f, & (t,x)\in (0,T) \times \T^2, \\
\div_x U(t,x) = 0, & (t,x) \in (0,T) \times \T^2, \\
\int_{\T^2} U(t,x) \ud x = 0, & t\in (0,T), \\
f(0,x,v)=f_0(x,v), & (x,v) \in \T^2 \times \R^2,  
\end{array} \right.
\label{eq:VS}
\end{equation} where $\lambda>0$ is a friction coefficient and
\begin{equation}
j_f(t,x) := \int_{\R^2} vf(t,x,v) \ud v. 
\label{eq:currentdendisty}
\end{equation} This is a control system  in which the state is the distribution function $f(t,x,v)$ and the control is the source term $1_\omega(x) G(t,x,v)$, located in $[0,T] \times \omega \times \R^2$. \par
The Vlasov-Stokes system studied in this paper is an adaptation to the case of the 2-dimensional torus of the system obtained by P.E. Jabin and B. Perthame in \cite{Jabin}. System (\ref{eq:VS}) is a kinetic-fluid model describing the behaviour of a large cloud of particles, represented by the distribution function $f(t,x,v)$, interacting with an incompressible fluid, whose velocity field is given by $U(t,x)$, under the hypothesis that the effects of convection are negligible. The quantity $f(t,x,v)\ud x \ud v$ can be interpreted as the number of particles at time $t$ whose position is close to $x$ and whose velocity is close to $v$. This model is especially convenient when describing sprays and aerosols, bubbly flows or suspension and sedimentation phenomena. This system is also important in biological applications, such as the transport in the respiratory tract (see Section 1.2.2 for more details).

\subsection{Main result}

We are interested in the controllability properties of system (\ref{eq:VS}), by means of an internal control. The controllability problem that we want to solve is the following. Given $f_0$ and $f_1$ in a suitable function space and given $T>0$, is it possible to find a control $G$ steering the solution of (\ref{eq:VS}) from $f_0$ to $f_1$, in time $T$? In other words, we want to find $G$ such that
\begin{equation}
f(T,x,v) = f_1(x,v), \quad \forall \,  (x,v) \in  \T^2  \times \R^2.
\label{eq:finalstate}
\end{equation} Let us observe that a natural constraint regarding the control $G$ is in order. Indeed, since the Vlasov-Stokes system preserves the total mass when $G\equiv 0$, i.e.,
\begin{equation*}
\int_{\T^2} \int_{\R^2} f(t,x,v) \ud x \ud v = \int_{\T^2} \int_{\R^2} f_0(x,v) \ud x \ud v, \quad \forall t\in[0,T],
\end{equation*} we shall prescribe the condition
\begin{equation*}
\int_{\T^2} \int_{\R^2} G(t,x,v) \ud x \ud v = 0, \quad \forall t\in [0,T].
\end{equation*} More precisely, we obtain the following controllability result.

\begin{thm}
Let $T>0$, $\gamma>2$, $\lambda = 1$ and let $\omega$ be an arbitrary non empty open subset of $\T^2$. There exists $\epsilon>0$ such that for every $f_0, f_1 \in \C^1(\T^2\times \R^2) \cap W^{1,\infty}(\T^2 \times \R^2)$ satisfying that 
\begin{align}
& \int_{\T^2} \int_{\R^2}f_0(x,v) \ud x \ud v = \int_{\T^2} \int_{\R^2} f_1(x,v) \ud x \ud v,  
\nonumber \\
& \int_{\T^2} \int_{\R^2} v f_0(x,v) \ud x \ud v = \int_{\T^2} \int_{\R^2} v f_1(x,v) \ud x \ud v = 0, 
\label{eq:compatibility}
\end{align} and that, for $i = 0,1$, 
\begin{align}
&\|f_i\|_{\C^1(\T^2\times\R^2)} + \| (1+|v|)^{\gamma+2} f_i\|_{\C^0(\T^2\times \R^2)} \leq \epsilon, \label{eq:smalldata} \\
&\exists \kappa>0, \quad \left(|\nabla_x f_i| + |\nabla_v f_i| \right)(x,v) \leq \frac{\kappa}{(1+|v|)^{\gamma+1}}, \, \forall (x,v) \in \T^2\times\R^2, \label{eq:conditionforuniqueness}  
\end{align} there exists a control $G \in \C^0([0,T]\times \T^2 \times \R^2)$ such that the solution of (\ref{eq:VS}) with $f_{|t=0} = f_0$ exists, is unique and satisfies (\ref{eq:finalstate}).
\label{thm:controllability}  
\end{thm}

\begin{remark}
Condition (\ref{eq:compatibility}) in the previous statement can be seen as a natural compatibility condition needed for the well-posedness of the Stokes system with sources $j_{f_0}$ and $j_{f_1}$.
\end{remark}

\subsection{Previous work}
\subsubsection{The controllability of kinetic equations}
There exist some results on the controllability of nonlinear kinetic equations. The first one was obtained by O. Glass for the Vlasov-Poisson system on the torus (see \cite{Glass}). The strategy of this work consists on the construction of a reference solution, in the spirit of the return method introduced by J.-M. Coron (see 1.3.2 below). This allows to conclude the existence and uniqueness of a controlled solution by means of the Leray-Schauder theorem. \par
The strategy of \cite{Glass} permits to obtain two types of results:
\begin{itemize}
\item In dimension 2, with arbitrary control region, one can obtain a local controllability result, i.e., a small-data result. 
\item In any dimension and with a geometric assumption, precisely that the control region $\omega \subset \T^n$ contains a hyperplane of $\R^n$ by the canonical surjection, one can obtain a global exact controllability result, i.e., an arbitrary-data result. However, the use of the invariant scaling of the Vlasov-Poisson system is crucial in this case.    
\end{itemize} 
This strategy was later extended in \cite{GDHK1} by O. Glass and D. Han-Kwan to the Vlasov-Poisson system under external and Lorentz forces. The authors obtain both local and global exact controllability results in the case of bounded external forces, which requires some new ideas to construct the reference trajectories. Precisely, the authors exploit the fact that the dynamics under the external force and without it are similar in small time. In the case of Lorentz forces, a precise knowledge of the magnetic field and a geometric control condition in the spirit of \cite{Bardos} allow to obtain a local exact controllability result. The functional framework of \cite{Glass, GDHK1} is the one given by the classical solution of the Vlasov-Poisson system, that is, some appropriate H\"{o}lder spaces. \par

A remarkably different strategy has been developed by \cite{GDHK2} in the context of the Vlasov-Maxwell system. In this case, the authors combine the classical strategy described previously with some controllability results for the Maxwell system, under the geometric control condition of \cite{Bardos}. They also obtain a local result for $\omega$ containing a hyperplane, using the convergence towards the Vlasov-Poisson system under a certain regime.  

\subsubsection{A short review on the Vlasov-Stokes system}
The Vlasov-Stokes system has been rigorously derived from the dynamics of a system of particles in a fluid by P. E Jabin and B. Perthame in \cite{Jabin}. The system derived in this result, using the method of reflections and the dipole approximation, is set in the whole phase space $\R^3 \times \R^3$. In this setting, some regularity results can be found in \cite{GJP}. The limit when $\lambda \rightarrow \infty$ has been studied in \cite{Jabinfriction, JabinM3AS}. Moreover, a major feature emphasised in these works is that friction plays a very important role in the dynamics of the Vlasov-Stokes system. More precisely, friction entails the dissipation of the kinetic energy. Consequently, as it has been proven by P. E. Jabin in \cite{Jabinasymp}, when $t \rightarrow \infty$, we recover a macroscopic limit of the form $\ovl{\rho}(x) \delta_{v=0}$. On the other hand, very little information concerning the density $\ovl{\rho}$ is known.  \par 
The non-stationary Vlasov-Stokes system on a domain with boundary has been considered by K. Hamdache in \cite{Hamdache}. The author gives a well-posedness result in Sobolev spaces in the case of Dirichlet boundary condition for the velocity field and specular reflexion boundary conditions for the distribution function. \par
A derivation of a model considering also the effects of convection has been obtained by L. Desvillettes, F. Golse and V. Ricci in \cite{Golse}.\par 
For more concrete biological models, let us cite \cite{Boudin} and the references therein.  
\subsection{Strategy of the proof}

\subsubsection{Obstructions to controllability}
Since Theorem \ref{thm:controllability} is of local nature around the steady state $(f,U,p)=(0,0,0)$, a first step to achieve its proof could be the use of the linear test (see \cite{Coron}). Following the classical scheme, the controllability of the linearised system around the trivial trajectory and the classical inverse mapping theorem between proper functional spaces would imply the controllability of the nonlinear system (\ref{eq:VS}). \par 
Indeed, the formal linearised equation around the trajectory $(f,U,p) = (0,0,0)$ is
\begin{equation}
\left\{ \begin{array}{ll}
\d_t F + v\cdot \nabla_x F - v\cdot \nabla_v F - 2 F = 1_{\omega}(x)\tilde{G}, \\
F(0,x,v) = f_0(x,v),
\end{array} \right.
\label{eq:linearised}
\end{equation} which is a transport equation with friction. By the method of characteristics, we can give an explicit solution of (\ref{eq:linearised}), which writes
\begin{equation}
F(t,x,v) = e^{2t}f_0(x+ (1-e^{t})v, e^t v) + \int_0^t e^{2(t-s)} (1_{\omega}\tilde{G})(s,x+(1-e^{t-s})v, e^{t-s}v) \ud s.
\label{eq:explicit}
\end{equation} As pointed out in \cite{Glass}, there exist two obstructions for controllability, which are
\begin{description}
\item[Small velocities] A certain $(x,v)\in \T^2 \times \R^2$ can have a "good direction" with respect to the control region $\omega$, in the sense that $x+(1-e^{-t})v$ meets $\omega$ at some time. However, if $|v|$ is not sufficiently large, the trajectory of the characteristic beginning at this point would possibly not reach $\omega$ before a fixed time. In our case, the effects of friction could enhance this difficulty.
\item[Large velocities] The obstruction concerning large velocities is of geometrical nature. There exist some "bad directions" with respect to $\omega$, in the sense that a characteristic curve parting from $(x,v)\in \T^2\times \R^2$ would never reach $\omega$, no matter how large $|v|$ is. 
\end{description}  As a result of this, and considering again equation (\ref{eq:explicit}), we deduce that the linearised system is not controllable in general.

\subsubsection{The return method}
In order to circumvent these difficulties, we use the return method, due to J.-M. Coron.  \par
The idea of this method, in the case under study, is to construct a reference trajectory $(\ovl{f}, \ovl{U}, \ovl{p})$ starting from $(0,0,0)$ and coming back to $(0,0,0)$ at some fixed time in such a way the linearised system around it is controllable. This method allows to avoid the problems discussed in the previous section. \par 
We refer to \cite{Coron, GlassBourbaki} for presentations and examples on the return method.

\subsubsection{Strategy of the proof of Theorem \ref{thm:controllability}}
The strategy of this work follows very closely the scheme of \cite{Glass, GDHK1}. More precisely, it relies on two ingredients.
\begin{description}
\item[Step 1] We build a reference solution $(\ovl{f},\ovl{U},\ovl{p})$ of system (\ref{eq:VS}) with a control $\ovl{G}$, located in $\omega$, starting from $(0,0,0)$ and arriving at $(0,0,0)$ outside $\omega$  at time $T>0$ and such that the characteristics associated to the field $-v + \ovl{U}$ meet $\omega$ before $T>0$.
\item[Step 2] We build a solution $(f,U,p)$ close to $(\ovl{f},\ovl{U},\ovl{p})$ parting from $(f_0,U_0,p_0)$ and arriving at $(0,0,0)$ outside $\omega$ at time $T>0$. This can be done by means of a fixed-point argument involving an absorption operator in the control region.
\end{description} Furthermore, let us note that in the proof of Theorem \ref{thm:controllability} we can assume that 
\begin{equation}
f_1(x,v) = 0, \quad \forall (x,v) \in (\T^2 \setminus \omega) \times \R^2.
\label{eq:f1simplified}
\end{equation} To justify this assumption, we observe that, if $f$ is solution of (\ref{eq:VS}), then the functions
\begin{align*}
& \tilde{f}(t,x,v) := f(T-t,x,-v),  & \tilde{G}(t,x,v) := G(T-t,x,-v), \\
& \tilde{U}(t,x) := U(T-t,x),  & \tilde{p}(t,x) = p(T-t,x), 
\end{align*} for every $(t,x,v) \in [0,T] \times \T^2 \times \R^2$, satisfy the backwards Vlasov-Stokes system
\begin{equation}
\left\{  \begin{array}{ll}
\partial_t \tilde{f} + v\cdot \nabla_x \tilde{f} + \lambda \div_v\left[ (\tilde{U}+v) \tilde{f}  \right]= 1_{\omega}(x)\tilde{G}, & (t,x,v) \in (0,T) \times \T^2 \times \R^2, \\
-\Delta_x \tilde{U} + \nabla_x \tilde{p} = -j_{\tilde{f}}, & (t,x)\in (0,T) \times \T^2, \\
\div_x \tilde{U}(t,x) = 0, & (t,x) \in (0,T) \times \T^2, \\
\int_{\T^2} \tilde{U}(t,x) \ud x =0, & t\in [0,T], \\
\tilde{f}(T,x,v)=f_0(x,v), & (x,v) \in \T^2 \times \R^2,  
\end{array} \right.
\label{eq:VSbackwards}
\end{equation} Consequently, given $f_0, f_1$ as in Theorem \ref{thm:controllability}, it is sufficient to consider
\begin{itemize}
\item $f_0$ as initial datum and $\hat{f}_1$ satisfying (\ref{eq:f1simplified}) as a final state,
\item $f_1(x,-v)$ as initial datum and again $\hat{f}_1$ satisfying (\ref{eq:f1simplified}) as a final state.
\end{itemize} If we are able to solve these problems, a simple composition of them gives a solution with initial datum $f_0$ and final state $f_1$, We observe, as it will be clear from the proofs, that (\ref{eq:VSbackwards}) can be treated like the forward problem without significant modifications. As a consequence, we shall only treat specifically the forward problem with final state satisfying (\ref{eq:f1simplified}).

\subsubsection{Notation}
Let $T>0$. We denote $Q_T:=[0,T] \times \T^2 \times \R^2$ and $\Omega_T := [0,T] \times \T^2$. If $\Omega$ is a domain, for any $\sigma \in (0,1)$, $\C^{0,\sigma}_b(\Omega)$ denotes the space of bounded $\sigma-$H\"{o}lder functions, equipped with the norm
\begin{equation}
\| f \|_{\C^{0,\sigma}_b(\Omega)} := \|f\|_{L^{\infty}(\Omega)} + \sup_{(t,x,v) \not = (t',x',v')} \frac{|f(t,x,v) - f(t',x',v')|}{|(t,x,v) - (t',x',v')|^{\sigma}}.
\end{equation} We shall also use the Sobolev spaces $W^{m,p}(\Omega)$, with $m\in \N^*$ and $p\in [1,\infty]$ (see the Appendix B for more details). If $X$ is a Banach space, we will sometimes use, for simplicity, the notations $L^p_tX_x$ or $\C^0_tX_x$ to refer to $L^p(0,T;X)$ or $\C^0([0,T];X)$. \par
For $x\in \T^2$ and $r>0$, we denote by $B(x,r)$ the open ball in $\T^2$ with centre $x$ and radius $r$. Analogously, $S(x,r) = \partial B(x,r)$. We denote by $B_{\R^2}$ and $B_{\Sph^1}$ the balls in different settings. We will also admit that $\int_{\T^2} \ud x = 1$ without specifying the normalisation.\par
In dimension two, given a vector field $V\in \C^1(\R^2;\R^2)$, with $V(x) = (V_1,V_2)(x)$ we recall the usual operator
\begin{equation*}
\curl V(x) := \d_1 V_2 (x)-\d_2 V_1 (x). 
\end{equation*} Given a function $\phi\in \C^1(\R^2;\R)$, we recall
\begin{equation*}
\nabla^{\perp} \phi(x) := \begin{pmatrix} -\d_2 \phi(x) \\ \d_1 \phi(x) \end{pmatrix}. 
\end{equation*}  

\subsubsection{Structure of the article}
In Section 2 we set some features of the characteristic equations. In Section 3 we construct the reference trajectory, treating separately the large velocities and the low ones. In Section 4 we define the fixed-point operator and we show that it has a fixed point. We prove next that this fixed point is the unique solution of system (\ref{eq:VS}) within a certain class. In Section 5 we show that this solution satisfies the controllability property (\ref{eq:finalstate}), which ends the proof of Theorem \ref{thm:controllability}. In section 6 we give some conclusions and comments. Finally, we gather in the Appendix some auxiliary results about harmonic approximation and the Stokes system.

\section{Some remarks on the characteristic equations}
Let be given a fixed $U(t,x)$. Let $s,t \in [0,T]$, $(x,v)\in \T^n\times \R^n$. We denote by $(X(t,s,x,v),V(t,s,x,v))$ the characteristics associated with the field $-v + U(t,x)$, i.e., the solution of the system
\begin{equation}
\left\{ \begin{array}{ll}
\frac{\ud}{\ud t} \begin{pmatrix} X \\ V \end{pmatrix} = \begin{pmatrix}   V(t)  \\ - V(t) + U(t,X) \end{pmatrix}, \\
\begin{pmatrix}   X \\ V \end{pmatrix}_{|t=s} = \begin{pmatrix} x \\ v \end{pmatrix}. 
\end{array} \right.
\label{eq:characteristcsystem}
\end{equation} We observe that if $U\in \C^0([0,T]; \C^1(\T^2;\R^2))$, system (\ref{eq:characteristcsystem}) has a unique solution, thanks to the Cauchy-Lipschitz theorem. Moreover, one has the explicit formulae 
\begin{equation}
\left\{ \begin{array}{ll}
X(t,s,x,v) = x+(1-e^{-t+s})v + \int_s^t \int_s^{t'} e^{\tau - t'} U(\tau, X(\tau,s,x,v)) \ud \tau \ud t', \\
V(t,s,x,v) = e^{-t+s}v + \int_s^t e^{\tau - t} U(\tau, X(\tau,s,x,v))\ud \tau.
\end{array} \right.
\label{eq:characteristics}
\end{equation} 

Using the method of characteristics, given an initial datum $f_0 \in \C^0(\T^2 \times \R^2)$, the solution of the transport equation with friction 
\begin{equation}
\left\{ \begin{array}{ll}
\d_t f + v\cdot \nabla_x f + \div_v\left[(U-v)f  \right] = 0, & (t,x,v)\in (0,T)\times \T^2 \times \R^2, \\
f(0,x,v) = f_0(x,v), & (x,v) \in \T^2 \times \R^2, 
\end{array} \right.
\end{equation} has the explicit solution
\begin{equation}
f(t,x,v) = e^{2t}f_0((X,V)(0,t,x,v)),
\label{eq:explicitsolution}
\end{equation} where $(X,V)$ are given by (\ref{eq:characteristics}).

The following result is an adaptation of \cite[Lemma 1, p. 337]{Glass} to the case with friction. It will be used to obtain some H\"{o}lder estimates in Section 4. \par 

\begin{lemma}
Let $U\in \C^0([0,T]; \C^1 (\T^2;\R^2))$ Then, the characteristics associated to the field $-v + U$ satisfy that for some $C = C(T,\|U\|_{\C^{0,1}_{t,x}})>0$,
\begin{eqnarray}
&& |(X,V)(t,s,x,v) - (X,V)(t',s',x',v')| \nonumber \\
&& \quad \quad \quad \quad \quad \quad \leq C (1+|v|)|(t,s,x,v) - (t',s',x',v')|, \nonumber 
\end{eqnarray} whenever $(t,s,x,v),(t',s',x',v') \in [0,T] \times \T^2 \times \R^2$, with $|v-v'|<1$.
\label{lemma:Gronwall}
\end{lemma} 

\begin{proof}
We shall divide the proof in four cases. 
\begin{description}
\item[Step 1] Assume $s=s', \, x=x', \, v=v'$. We can suppose that $t'\leq t$. Then, using (\ref{eq:characteristics}), we write
\begin{eqnarray}
&& V(t,s,x,v) - V(t',s,x,v) \nonumber \\
&& \quad \quad = (e^{-t} - e^{-t'})e^s v + \int_s^t e^{\tau - t} U(\tau,X(\tau,s,x,v)) \ud \tau  \nonumber \\
&& \quad \quad \quad \quad \quad \quad - \int_s^{t'} e^{\tau - t'} U(\tau,X(\tau,s,x,v)) \ud \tau \nonumber \\
&& \quad \quad = (e^{-t} - e^{-t'})e^s v + \int_s^t e^{\tau}(e^{-t} - e^{-t'})U(\tau,X(\tau,s,x,v)) \ud \tau \nonumber \\
&& \quad \quad \quad \quad \quad \quad  -\int_{t'}^t e^{\tau - t'}U(\tau, X(\tau,s,x,v)) \ud \tau.\nonumber 
\end{eqnarray} This yields,
\begin{eqnarray}
&& |V(t,s,x,v) - V(t',s,x,v)| \nonumber \\
&& \quad \quad \leq |e^{-t} - e^{-t'}| \left( e^T|v| + Te^T \|U\|_{\C^{0,1}_{t,x}} \right) + |t-t'|e^T \|U\|_{\C^{0,1}_{t,x}}) \nonumber \\
&& \quad \quad \leq C(T)(1+|v|)(1 + \|U\|_{\C^{0,1}_{t,x}}) |t-t'|.   \nonumber
\end{eqnarray} The same argument gives, through (\ref{eq:characteristics}),
\begin{equation*}
|X(t',s,x,v) - X(t,s,x,v)| \leq C(T)(1+|v|)(1 + \|U\|_{\C^{0,1}_{t,x}}) |t-t'|. 
\end{equation*}
\item[Step 2] Assume $t=t', \, s=s'$. Then, again by (\ref{eq:characteristics}), we have
\begin{eqnarray}
&& X(t,s,x,v) - X(t,s,x',v') \nonumber \\
&& \quad \quad = (x-x') + (1-e^{-t+s})(v-v') \nonumber \\
&& \quad \quad \quad \quad + \int_s^t \int_s^{\sigma} e^{\sigma-\tau} \left( U(\tau,X(\tau,s,x,v)) - U(\tau,X(\tau,s,x',v') \right)  \ud \tau \ud \sigma, \nonumber  
\end{eqnarray} which gives, since $U$ is Lipschitz in $x$,
\begin{eqnarray}
&& |X(t,s,x,v) - X(t,s,x',v')| \nonumber \\
&& \quad \quad \leq  (e^T + 1)\left( |x-x'| + |v-v'|  \right) \nonumber \\
&& \quad \quad \quad \quad + Te^T \| U \|_{\C^{0,1}_{t,x}} \int_s^t \left| X(\tau,s,x,v) - X(\tau,s,x',v') \right| \ud \tau. \nonumber 
\end{eqnarray} Then, by Gronwall's lemma, we obtain
\begin{eqnarray}
&& |X(t,s,x,v) - X(t,s,x',v')| \nonumber \\
&& \quad \quad \leq (e^T + 1 )\left( |x-x'|+|v-v'| \right)(1 + Te^T\|U\|_{\C^{0,1}_{t,x}})Te^{T^2e^T\|U\|_{\C^{0,1}_{t,x}}} \nonumber \\
&& \quad \quad \leq C(T) \left( |x-x'|+|v-v'| \right) (1 + \|U\|_{\C^{0,1}_{t,x}})e^{T^2e^T\|U\|_{\C^{0,1}_{t,x}}}. \nonumber 
\end{eqnarray} This allows to obtain an analogue estimate for $V$.

\item[Step 3] Assume that $t=t', \, x=x', \, v=v'$. We observe that
\begin{equation*}
(X,V)(t,s',x,v) = (X,V)(t,s,X(s,s',x,v),V(s,s',x,v)).
\end{equation*} Thus, Step 2 allows to write
\begin{eqnarray}
&& |(X,V)(t,s,x,v) - (X,V)(t,s',x,v) |\nonumber \\
&& \quad \leq C(T) \left( |x-X(s,s',x,v)|+|v-V(s,s',x,v)| \right) (1 + \|U\|_{\C^{0,1}_{t,x}})e^{T^2e^T\|U\|_{\C^{0,1}_{t,x}}}. \nonumber 
\end{eqnarray} On the other hand, putting $ (x,v) = (X,V)(s',s',x,v)$, Step 1 allows to write
\begin{eqnarray}
&& |x-X(s,s',x,v)|+|v-V(s,s',x,v)| \nonumber \\
&& \quad \quad \leq |s-s'|C(T)(1 + |v|) (1 + \| U\|_{\C^{0,1}_{t,x}}), \nonumber 
\end{eqnarray} which together with the previous inequality gives
\begin{eqnarray}
&& |(X,V)(t,s,x,v) - (X,V)(t,s',x,v) |\nonumber \\
&& \quad \leq C(T) (1 + |v|) (1 + \| U\|_{\C^{0,1}_{t,x}})^2e^{T^2e^T\|U\|} |s-s'|. \nonumber
\end{eqnarray}
\item[Step 4] In the general case, we write
\begin{align*}
& (X,V)(t,s,x,v) - (X,V)(t',s',x',v') \\
& \quad \quad = (X,V)(t,s,x,v) - (X,V)(t,s',x,v) \\
& \quad \quad \quad + (X,V)(t,s',x,v) - (X,V)(t,s',x',v') \\
& \quad \quad \quad + (X,V)(t,s',x',v') - (X,V)(t',s',x',v'),  
\end{align*} which allows to use the previous estimates to conclude.

\end{description}

\end{proof}

\section{Construction of a reference trajectory}

\subsection{Large velocities}
The key result for the treatment of large velocities is the following one, which is an adaptation of \cite[Proposition 1, p. 340]{Glass} to the friction case. We will need some results on harmonic approximation, gathered in the Appendix A.

\begin{proposition}
Let $\tau >0$. Given $x_0 \in \T^2$ and $r_0>0$ a sufficiently small number, there exist $U \in \C^{\infty}([0,\tau]\times \T^2;\R^2)$ and $\underline{m} >0$ such that 
\begin{eqnarray}
&& \curl_x U (t,x) = 0, \quad   \forall (t,x) \in [0,\tau] \times \left(  \T^2 \setminus \ovl{B}(x_0,r_0/10)  \right), \label{eq:psiharmonic}  \\
&& \supp U \subset (0,\tau) \times \T^2, \label{eq:suppoertpsi} \\
&& \int_{\T^2} U(t,x) \ud x = 0, \quad \forall t\in[0,\tau]. \label{eq:integralenulleHV}
\end{eqnarray} Moreover, the characteristics $(X,V)$ associated to the field $-v + U$ satisfy that, for every $m \geq \underline{m}$,
\begin{eqnarray}
&& \forall (x,v) \in \T^2 \times \R^2 \textrm{ with } |v| \geq m, \exists t \in \left( \frac{\tau}{4}, \frac{3\tau}{4} \right) \textrm{ such that } \nonumber \\
&& X(t,0,x,v) \in B \left(x_0, \frac{r_0}{4}  \right) \textrm{ and } |V(t,0,x,v) | \geq \frac{m}{2e^{\tau}}. \label{eq:largevelocitycharacteristics}
\end{eqnarray}
\label{proposition:highvelocities}
\end{proposition}

Following O. Glass (see \cite{Glass}), we characterise the bad directions in the following sense.
\begin{definition}[Bad directions]
Given $x_0 \in \T^2$ and $r_0>0$ a small number, $e\in \Sph^1$ is a bad direction if 
\begin{equation*}
\left\{ (x+te); \, t\in \R^+ \right\} \cap B(x_0, r_0/4) = \emptyset,
\end{equation*} in the sense of $\T^2$.
\label{definition:baddirections} 
\end{definition}

\begin{remark}
It can be shown, thanks to B\'{e}zout's theorem, that for any $x_0 \in \T^2$ and $r_0>0$ small, there exists only a finite number of such bad directions, namely $\left\{ e_1, \dots e_N  \right\}$ (see \cite[Appendix A, p. 373]{Glass}). 
\end{remark}

\begin{proof}[Proof of Proposition \ref{proposition:highvelocities}]
Given $\tau >0$ and $N\in \N^*$, the number of bad directions, let us define 
\begin{equation}
t_j := \frac{\tau}{4} + \frac{j\tau}{2(N+1)}, \quad \forall j \in \left\{ i + \frac{k}{4}; \,\, i=0,\dots , N, \, k=0,1,2,3 \right\}.
\label{eq:tis}
\end{equation} We consider
\begin{equation} 
\eta \in \C^{\infty}_c(0,1) \textrm{ such that } 0\leq \eta \leq 1, \textrm{ and } \int_0^1 \eta(t) \ud t = 1.
\label{eq:eta}
\end{equation} Let $A,\nu >0 $ so that 
\begin{equation}
\nu <\frac{\tau}{8(N+1)}, \quad A >  e^{\frac{\tau}{8(N+1)}} \left( \frac{12(N+1)}{\tau} + 2  + \frac{\tau}{4(N+1)} \right).
\label{eq:choiceofA}
\end{equation} Let $\epsilon>0$ to be chosen later on and set, according to Proposition \ref{proposition:highvelocitiesapproximation}, the following vector field
\begin{equation}
\left\{ \begin{array}{ll}
U(t,x) = \frac{A}{\nu} \eta \left( \frac{t - t_{i + \frac{1}{4}}}{\nu}  \right) \nabla^{\perp} \theta^i(x), & (t,x) \in [t_{i+\frac{1}{4}},t_{i + \frac{1}{2}}] \times \T^2, \\
U(t,x) = 0, & \textrm{otherwise}.
\end{array} \right.
\label{eq:choiceoflargevelocityfield}
\end{equation} With this definition, we readily have (\ref{eq:psiharmonic}) and (\ref{eq:suppoertpsi}), using (\ref{eq:largephiharmonic}) and (\ref{eq:eta}). Condition (\ref{eq:integralenulleHV}) follows from the definition of $U$ above. We have to show that (\ref{eq:largevelocitycharacteristics}) is satisfied with this construction. \par 
To prove the second point of (\ref{eq:largevelocitycharacteristics}), (\ref{eq:characteristics}) yields
\begin{eqnarray}
|V(t,0,x,v)| &\geq & e^{-t}|v| - \tau \|U\|_{\C^0_{t,x}} \nonumber \\
& \geq & e^{-\tau}m - \tau \|U\|_{\C^0_{t,x}} \,  \geq e^{-\tau}\frac{m}{2}, \nonumber 
\end{eqnarray} provided that $m$ is large enough. \par
To show the first part of (\ref{eq:largevelocitycharacteristics}), we distinguish several cases, according to $(x,v) \in \T^2 \times \R^2$. More precisely, if $e= \frac{v}{|v|} \in \Sph^1$, we shall show that
\begin{enumerate}
\item if $x\in\T^2$ and $e \in \Sph^1\setminus \left\{ e_1,\dots, e_N \right\}$, with $|v|$ large enough, then (\ref{eq:largevelocitycharacteristics}) follows by comparison with the free transport without friction,
\item if $x \in B(x_0, \frac{r_0}{5}) + \R e_i $ and $ e \in B_{\Sph^1}(e_i;\epsilon_1)$, for some $i\in \left\{ 1,\dots,N \right\}$ and $\epsilon_1>0$, then (\ref{eq:largevelocitycharacteristics}) still holds if $|v|$ is large enough,
\item if $x\in \T^2$, $ e \in B_{\Sph^1}(e_i;\epsilon_2)$ for some $i\in \left\{ 1,\dots,N \right\}$ and $\epsilon_2>0$ small enough, then for $|v|$ large enough, we shall deduce that $(X,V)(\sigma,0,x,v)$, with some $\sigma\in[t_i,t_{i+\frac{3}{4}}],$ satisfies the hypotheses of the case (2).
\end{enumerate} Let us now treat in detail the three cases below.

\begin{description}
\item[First case] Let $e\in \Sph^1 \setminus \left\{ e_1,\dots, e_N \right\}$. Since $e$ is not a bad direction, from Definition \ref{definition:baddirections}, there exists  $m>0$ such that if $|v|\geq m$, then, 
\begin{equation*}
x + t|v|e \in B\left( x_0, \frac{r_0}{4} \right),
\end{equation*}for a certain $t\in [t_0,t_1]$. In particular, $\exists t_{trans} \in [t_0,t_1]$ such that 
\begin{equation}
x + t_{trans}(1+m)e \in B\left( x_0, \frac{r_0}{4} \right).
\label{eq:ttilde}
\end{equation} We shall prove next that, by augmenting the minimal speed required, we can conclude in the friction case. Indeed, consider 
\begin{equation*}
m_0:= \frac{t_1 (1 + m)}{1 - e^{-t_1}}. 
\end{equation*} and for any $v^{\flat} \in \R^2$ with $|v^{\flat}|\geq m_0$, set the function
\begin{equation*}
t \mapsto f_{\flat}(t):=(1 - e^{-t})|v^{\flat}|,
\end{equation*} which is continuous from $\R^+$ to $\R^+$. Since $f_{\flat}(0) = 0 $ and $f_{\flat}(t_1) \geq t_1 (1+m)$, by the intermediate value theorem, $\exists t^*\in(0,t_1)$ such that $f_{\flat}(t^*) = t_{trans}(1+m)$. Whence, by (\ref{eq:ttilde}),
\begin{equation}
x + (1-e^{-t^*})|v^{\flat}|e = x + f_{\flat}(t^*)e=x + t_{trans}(1+m)e \in B\left(x_0, \frac{r_0}{4} \right).
\end{equation}   This shows (\ref{eq:largevelocitycharacteristics}) in this case, since $U(t,x) =0 $ for $t\in [t_0,t_1]$, and thus,  $X(t^*,0,x,v) = x + (1-e^{-t^*})v$.

\item[Second case]
Let us suppose that $x \in B(x_0, \frac{r_0}{5}) + \R e_i$ for some $i \in \left\{ 1,\dots,N\right\}$. Let $e \in \V_1(e_i):=B_{\Sph^1}(e_i;\epsilon_1)$. If $\epsilon_1>0$ is small enough and $|v|\geq m_1 $ is large enough, then there exists  $0< s\leq t \leq \frac{C}{|v|}$, for some constant $C>0$, independent of $(x,e) \in \T^2 \times \Sph^1$, such that
\begin{equation*}
x + (1-e^{-t+s})|v|e \in B\left(x_0, \frac{9r_0}{40}  \right).   
\end{equation*} To justify this, we point out that this holds for the free transport (see \cite[p.375]{Glass}), which implies, by a similar argument as before, that this also holds in the friction case. Whence, from (\ref{eq:characteristics}) and for any $s',s''\in [0,\tau]$,   
\begin{equation*}
|X(s'',s',x,v) - x - (1-e^{-s' +s''})v | \leq C(T,\|U\|_{\C^0_{t,x}})|s'-s''| = O\left( \frac{1}{m_1} \right). 
\end{equation*}   Then, if $m_1$ is large enough, this entails in particular
\begin{equation*}
X(s,t_i,x,v) \in B\left( x_0, \frac{r_0}{4} \right), \textrm{ for some } s\in [t_i,t_{i+1}].
\end{equation*}

\item[Third case]
Let $x\in\T^2$, $|v|\geq m_2$, large enough. Let $e\in \mathcal{V}_2(e_i):=B_{\Sph^1}(e_i;\epsilon_2)$, for some $i\in \left\{1, \dots, N \right\}$ and $\epsilon_2>0$ small enough. We have to show that $\exists t \in [t_i, t_{i+\frac{3}{4}}]$ such that
\begin{eqnarray}
&& X(t,t_i,x,v) \in B\left(x_0, \frac{r_0}{5} \right) + \R e_i,
\label{eq:Xarrivesincase3} \\
&& \frac{V(t,t_i,x,v)}{|V(t,t_i,x,v)|} \in B_{\Sph^1}(e_i;\epsilon_1), \textrm{ with } |V(t,t_i,x,v)|\geq m_1.\label{eq:Varrivesincase3}
\end{eqnarray} Then, we can use the analysis of the second case to conclude. \par 
 We prove first (\ref{eq:Xarrivesincase3}) by means of the orthogonal projection on the direction $e_i^{\perp}$. Let 
\begin{equation*}
\begin{array}{llll}
P_{e_i^{\perp}} :& \R^2 & \rightarrow & \R \\
                 & v & \mapsto & P_{e^{\perp}_i}(v):= \langle v, e_i^{\perp} \rangle.
\end{array}
\end{equation*} We distinguish two cases. \par 
Firstly, let $v\in \R^2$ such that $|P_{e^{\perp}_i}(v)| > \frac{6(N+1)}{\tau}$. Then, for a large enough speed $|v| \geq m_2$, (\ref{eq:Xarrivesincase3}) is satisfied, as follows by comparison with the free transport case. \par
Secondly, if $|P_{e_i^{\perp}}(v)| < \frac{6(N+1)}{\tau}$, let us suppose that $X(t,t_i,x,v)$ does not meet $B(x_0, r_0/5) + \R e_i$ during $[t_{i+\frac{1}{4}},t_{i+\frac{1}{2}}]$. Then, by (\ref{eq:choiceoflargevelocityfield}) and Proposition \ref{proposition:highvelocitiesapproximation}, we have
\begin{eqnarray}
&& \left\| U(t,X(t,t_i,x,v)) - \frac{A}{\nu} \eta\left( \frac{t - t_{i + \frac{1}{4}}}{\nu} \right)e_i^{\perp} \right\|_{\C^0_{t,x}} \label{eq:UAmulessthanone} \\
&& \quad \quad \quad \quad \quad \quad \quad \leq  \frac{A}{\nu} \| \eta\|_{\C^0_t} \| \nabla^{\perp} \theta^i - e_i^{\perp} \|_{\C^0_x} < \frac{A\epsilon}{\nu}. \nonumber 
\end{eqnarray} We choose $\epsilon>0$ small enough so that $\frac{A\epsilon}{\nu} <1$. Then, by (\ref{eq:characteristics}), 
\begin{eqnarray}
&& |P_{e_i^{\perp}}(V(t_{i+\frac{1}{2}},t_i,x,v))| = |\langle V(t_{i+\frac{1}{2}},t_i,x,v), e_i^{\perp} \rangle | \nonumber \\
&& \quad \, \geq  -|P_{e_i^{\perp}}(v)|(t_{i+\frac{1}{2}} - t_i) + \left| \left \langle \int_{t_{i+\frac{1}{4}}}^{t_{i + \frac{1}{2}}} e^{s - t_{i+\frac{1}{2}}} U(s, X(s,t_i,x,v) \ud s  , e_i^{\perp} \right \rangle  \right|. \label{eq:projectionbase}
\end{eqnarray} For the first term, by (\ref{eq:tis}) and the hypothesis on $P_{e_i^{\perp}}(v)$,
\begin{equation}
-|P_{e_i^{\perp}}(v)|(t_{i+\frac{1}{2}} - t_i) = - |P_{e_i^{\perp}}(v)| \frac{\tau}{4(N+1)} \geq  -\frac{3}{2}.
\label{eq:projection1}
\end{equation} For the second term, we write
\begin{align}
&\left| \left \langle \int_{t_{i+\frac{1}{4}}}^{t_{i + \frac{1}{2}}} e^{s - t_{i+\frac{1}{2}}} U(s, X(s,t_i,x,v)) \ud s  , e_i^{\perp} \right \rangle  \right| \nonumber \\
& \quad \quad \quad \quad \quad = \left| \langle \trans_1 + \trans_2, e^{\perp}_i \rangle \right| \geq -|\langle \trans_1, e^{\perp}_i \rangle | + |\langle \trans_2, e^{\perp}_i \rangle |,   
\label{eq:CSprojection} 
\end{align} with
\begin{eqnarray}
\trans_1 &:=& \int_{t_{i+\frac{1}{4}}}^{t_{i + \frac{1}{2}}} e^{s-t_{i+\frac{1}{2}}} \left(  U(s, X(s,t_i,x,v)) -   \frac{A}{\nu} \eta\left( \frac{s - t_{i+\frac{1}{4}}}{\nu}\right) e^{\perp}_i \right) \ud s, \nonumber \\
\trans_2 & := & \frac{A}{\nu} \int_{t_{i+\frac{1}{4}}}^{t_{i + \frac{1}{2}}}  e^{s-t_{i+\frac{1}{2}}} \eta\left( \frac{s -t_{i+\frac{1}{4}}}{\nu}\right) e_i^{\perp} \ud s . \nonumber   
\end{eqnarray} Then, using (\ref{eq:UAmulessthanone}) and (\ref{eq:tis}),
\begin{eqnarray}
|\langle \trans_1, e^{\perp}_i \rangle | &\leq & 2 (t_{i+\frac{1}{2}} - t_{i+\frac{i}{4}})  \left\| U - \frac{A}{\nu}\eta e_i^{\perp} \right\|_{\C^0_{t,x}} \nonumber \\
& \leq & \frac{\tau}{4(N+1)} \label{eq:T1}
\end{eqnarray} For the second term, using (\ref{eq:tis}) and (\ref{eq:eta}), we get
\begin{equation*}
|\langle \trans_2 , e_i^{\perp} \rangle| \geq e^{-\frac{\tau}{8(N+1)}} A  
\end{equation*} Combining last inequality with (\ref{eq:T1}), (\ref{eq:CSprojection}), (\ref{eq:projection1}) and (\ref{eq:projectionbase}), the choice of $A$ and $\nu$ in (\ref{eq:choiceofA}), this yields
\begin{equation*}
|P_{e_i^{\perp}}(V(t_{i+\frac{1}{2}},t,x,v))| \geq \frac{6(N+1)}{\tau}.
\end{equation*} Whence, by the previous point, $\exists t \in [t_{i+\frac{i}{2}}, t_{i+ \frac{3}{4}}]$ such that 
\begin{equation*}
X(t,t_i,x,v) \in B\left( x_0, \frac{r_0}{5} \right) + \R e_i, 
\end{equation*} which shows (\ref{eq:Xarrivesincase3}). We have to show next (\ref{eq:Varrivesincase3}). Indeed, from (\ref{eq:characteristics}), we deduce
\begin{equation*}
|V(t,t_i,x,v) - e^{-t+t_i}v| \leq |t-t_i| \| U\|_{\C^0_{t,x}}, 
\end{equation*}  which implies, by (\ref{eq:tis}), that
\begin{equation}
e^{-t+t_i}|v| - |V(t,t_i,x,v)| \leq \frac{3\tau \| U \|_{\C^0_{t,x}}}{8(N+1)} . 
\label{eq:largevelocityremainslarge}
\end{equation} Then, choosing $m_2$ large enough, we get the second point of (\ref{eq:Varrivesincase3}). On the other hand, we observe that
\begin{equation}
\left| \frac{V(t,t_i,x,v)}{|V(t,t_i,x,v)|} - e_i \right| \leq \left| \frac{V(t,t_i,x,v)}{|V(t,t_i,x,v)|} - \frac{v}{|v|}  \right| + \left| \frac{v}{|v|} - e_i \right|.
\label{eq:voisinagesphere}
\end{equation} By definition of $\V_2(e_i)$, $\left| \frac{v}{|v|} - e_i \right| < \epsilon_2$. With the other term, using (\ref{eq:characteristics}), we find
\begin{eqnarray}
&& \left| \frac{V(t,t_i,x,v)}{|V(t,t_i,x,v)|} - \frac{v}{|v|}  \right| \nonumber \\
&& = \left| \frac{e^{-t+t_i}|v| +\int_{t_i}^t e^{t-s}U(s,X(s,t_i,x,v))\ud s}{|V(t,t_i,x,v)|} - \frac{v}{|v|} \right|  \nonumber \\
&& \leq \left| \frac{e^{-t+t_i}v}{|V(t,t_i,x,v)|} - \frac{v}{|v|} \right| + \frac{\tau e^{\tau} \|U\|_{\C^0_{t,x}}}{|V(t,t_i,x,v)|} \nonumber \\
&& =  \frac{\left| e^{-t+t_i}|v| - |V(t,t_i,x,v)|  \right|}{|V(t,t_i,x,v)|}   + \frac{\tau e^{\tau} \|U\|_{\C^0_{t,x}}}{|V(t,t_i,x,v)|} \nonumber \\
&& \leq \frac{2 \tau e^{\tau}\| U \|_{\C^0_{t,x}}}{|V(t,t_i,x,v)|}. \nonumber
\end{eqnarray} This shows, by (\ref{eq:voisinagesphere}), that
\begin{equation*}
\left| \frac{V(t,t_i,x,v)}{|V(t,t_i,x,v)|} - e_i \right| \leq O \left(\frac{1}{|V(t,t_i,x,v)|} \right) + \epsilon_2,
\end{equation*} which entails, by (\ref{eq:largevelocityremainslarge}), that (\ref{eq:Varrivesincase3}) holds choosing $m_2$ large enough and $\epsilon_2>0$ small enough.
\end{description}

\end{proof}


\subsection{Low velocities}
The goal of this section is to prove the following result, which is the key ingredient for the treatment of low velocities. The main difficulty is to adapt the construction made in \cite[Proposition 2]{Glass} to the case with friction.

\begin{proposition}
Let $\tau >0$, $M>0$. Given $x_0 \in \T^2$ and $r_0>0$ a small positive number, there exists $U \in \C^{\infty}([0,\tau] \times \T^2;\R^2)$ satisfying
\begin{eqnarray}
&& \curl_x U(t,x) = 0, \quad   (t,x) \in [0,\tau] \times \left(  \T^2 \setminus \overline{B}(x_0,r_0)  \right), \label{eq:lowpsiharmonic}  \\
&& \supp U \subset (0,\tau) \times \T^2, \label{eq:lowsuppoertpsi} \\
&& \int_{\T^2} U(t,x) \ud x = 0, \quad \forall t\in [0,\tau], \label{eq:integralenulleBV}
\end{eqnarray} and such that, for some $\ovl{M}>0$, the characteristics associated to $-v + U$ satisfy that, for every $(x,v) \in \T^2\times B_{\R^2}(0,M)$, there exists $ t \in( 0,\tau)$ such that
\begin{equation}
V(t,0,x,v) \in B_{\R^2}(0, \ovl{M}) \setminus B_{\R^2}(0,M+1).   
\label{eq:acceleration} 
\end{equation}
\label{proposition:lowvelocities}
\end{proposition}

\begin{proof}
Let $\theta$ be as in Proposition \ref{proposition:auxiliarylowvelocities}. Since $\textrm{Ind}_{S(x_0,r_0)}(\nabla \theta) =0$, possibly after a continuous extension, we may define
\begin{equation}
m:= \inf_{x\in\T^2} |\nabla \theta (x)| >0.
\label{eq:theatnonnul}
\end{equation} Let $a, b \in \R$, to be chosen later on, and such that $c:= \frac{a}{b}$ is fixed. Let $\eta$ be as in (\ref{eq:eta}). Then, we define the field
\begin{equation*}
U(t,x) := a \eta(bt) \nabla^{\perp} \theta(x), \quad \forall (t,x) \in [0,\tau] \times \T^2.  
\end{equation*} Conditions (\ref{eq:lowpsiharmonic})--(\ref{eq:integralenulleBV}) follow from the definition of $U$ and the properties of $\theta$, given by Proposition \ref{proposition:auxiliarylowvelocities}. We have to show (\ref{eq:acceleration}).
Indeed, from (\ref{eq:characteristics}), we find that, for every $(x,v) \in \T^2 \times B_{\R^2}(0,M)$,
\begin{equation}
V(t,0,x,v) = e^{-t}v + a \int_0^t e^{\sigma - t} \eta(b\sigma) \nabla^{\perp} \theta(X(\sigma,0,x,v)) \ud \sigma.
\label{eq:lowvelocitiescharacteristics}
\end{equation} This gives, changing variables and using (\ref{eq:eta}),
\begin{eqnarray}
| V(t,0,x,v) - e^{-t}v| & \leq & \frac{a}{b} \| \nabla^{\perp} \theta\|_{\C^0_x} \int_0^{bt} e^{\frac{s}{b}-t} \eta(s) \ud s \nonumber \\
& \leq & c \|\nabla \theta \|_{\C^0_x}, \nonumber   
\end{eqnarray} whenever $t\leq \frac{1}{b}$. Consequently,
\begin{eqnarray}
|X(t,0,x,v) - x | & \leq & \int_0^t | V(s,0,x,v) - e^{-s}v | \ud s + \int_0^t |e^{-s}v| \ud s \nonumber \\
& \leq & \frac{c}{b}\|\theta\|_{\C^{1}_{x}} + \frac{M}{b}, \label{eq:majorationXlow}
\end{eqnarray} for every $t \leq \frac{1}{b}$. Thus, from (\ref{eq:lowvelocitiescharacteristics}),
\begin{eqnarray}
&& \left|V(\frac{1}{b},0,x,v) - e^{-\frac{1}{b}}v - c \nabla^{\perp} \theta(x) \right| \nonumber \\
&& = \left| a\int_0^{\frac{1}{b}} e^{s-\frac{1}{b}} \eta(bs) \nabla^{\perp}\theta (X(s,0,x,v)) \ud s - c \nabla^{\perp}\theta (x)  \right| \leq  I_1 + I_2, \nonumber
\end{eqnarray} with
\begin{align*}
&I_1 := \left| a\int_0^{\frac{1}{b}} e^{s - \frac{1}{b}} \eta(bs) \nabla^{\perp} \theta(X(s,0,x,v))\ud s - c\int_0^1 e^{\frac{\sigma-1}{b}} \eta (\sigma) \nabla^{\perp}\theta(X(\frac{\sigma}{2b},0,x,v)) \ud \sigma \right|, \nonumber \\
&I_2 := \left| c\int_0^1 e^{\frac{\sigma-1}{b}} \eta (\sigma) \nabla^{\perp}\theta(X(\frac{\sigma}{2b},0,x,v)) \ud \sigma - c\nabla^{\perp} \theta(x)   \right|. \nonumber
\end{align*} We note that the introduction of the term $X(\frac{\sigma}{2b},0,x,v)$ is intended to take into account the exponential in the first integral. \par 
For the first term, we write, changing variables,
\begin{eqnarray}
I_1 & = &  \left| a \int_0^{\frac{1}{b}} e^{s-\frac{1}{b}} \eta(bs) \left( \nabla^{\perp}\theta(X(s,0,x,v)) - \nabla^{\perp}\theta(X(\frac{s}{2},0,x,v))   \right) \ud s   \right| \nonumber \\
& \leq & a \| \theta \|_{\C^1_x} \int_0^{\frac{1}{b}}  e^{s-\frac{1}{b}} \eta(bs) \left|X(s,0,x,v) - X(\frac{s}{2},0,x,v) \right| \ud s \nonumber \\
& \leq & 2c \| \theta \|_{\C^1_x} \left( \frac{c}{b} \| \theta \|_{\C^1_x} + \frac{M}{b} \right) = O\left( \frac{1}{b} \right), \nonumber
\end{eqnarray} as a consequence of (\ref{eq:majorationXlow}). For the second term, by (\ref{eq:eta}), we have
\begin{eqnarray}
I_2 & \leq & \left| c\int_0^1 e^{\frac{\sigma -1}{b}}\eta(\sigma) \left( \nabla^{\perp}\theta ( X(\frac{\sigma}{2b},0,x,v)) - \nabla^{\perp}\theta(x) \right) \ud \sigma\right| \nonumber \\
&& \quad \quad \quad \quad + \left| c\int_0^1 \left( e^{\frac{\sigma - 1}{b}} - 1 \right) \eta(\sigma) \nabla^{\perp} \theta(x) \ud \sigma  \right| \nonumber \\
& \leq &  c \| \theta \|_{\C^1_x} \int_0^1 e^{\frac{\sigma-1}{b}} \eta(\sigma) \left|X(\frac{s}{2},0,x,v) - x \right| \ud s + c \| \theta \|_{\C^1_x} |e^{-\frac{1}{b}} - 1 |\nonumber \\
& \leq & c\| \theta \|_{\C^1_x} \left( \frac{c}{b}\| \theta \|_{\C^1_x} + \frac{M}{b} \right) + O\left(\frac{1}{b} \right) = O\left( \frac{1}{b} \right), \nonumber
\end{eqnarray} using again (\ref{eq:majorationXlow}). This allows to choose $b$ large enough so that
\begin{equation*}
\left|V(\frac{1}{b},0,x,v) - e^{-\frac{1}{b}}v - c \nabla^{\perp} \theta(x) \right| < \frac{1}{2}.
\end{equation*} Whence, by (\ref{eq:theatnonnul}),
\begin{eqnarray}
\frac{1}{2} & > & \left|V(\frac{1}{b},0,x,v) - e^{-\frac{1}{b}}v - c \nabla^{\perp} \theta(x) \right| \nonumber \\
&\geq &c|\nabla^{\perp} \theta(x)| - \left| V(\frac{1}{b},0,x,v) \right| - |v| \nonumber \\
& \geq & cm - \left| V(\frac{1}{b},0,x,v) \right|  - M, \nonumber     
\end{eqnarray} which gives, choosing $c:=\frac{2(M+1)}{m}$, 
\begin{equation*}
\left| V(\frac{1}{b},0,x,v) \right| > M + \frac{3}{2}.
\end{equation*} This concludes the proof.

\end{proof}

\subsection{Description of the reference trajectory}
Since $\omega$ is a nonempty open set in $\T^2$, there exist $x_0 \in \omega$ and $r_0>0$ such that
\begin{equation*}
B(x_0, 2r_0) \subset \omega.
\end{equation*} We can define a suitable vector field $\ovl{U}$ using the constructions made in the previous sections. Firstly, we apply Proposition \ref{proposition:highvelocities} with $\tau := \frac{T}{3}$, which gives a vector field $\ovl{U}_1$ and $\udl{m}_1>0$ such that (\ref{eq:largevelocitycharacteristics}) is verified. \par 
For reasons that will be clear in Section 5, we set the following parameters. Let
\begin{equation}
\alpha:=\max \left\{ T\| \ovl{U}_1\|_{\C^{0,1}_{t,x}} + \frac{5}{2}, \frac{\mathcal{C}_{r_0,T}\| \ovl{U}_1\|_{\C^{0,1}_{t,x}}}{4} \right\},  
\label{eq:choiceofalpha}
\end{equation} where $\mathcal{C}_{r_0,T}>0$ is a constant chosen large enough so that
\begin{equation}
\log \left( 1 + \frac{9r_0 }{\mathcal{C}_{r_0,T} \| \ovl{U}_1\|_{\C^{0,1}_{t,x}}}   \right) < \frac{T}{200}.
\label{eq:choiceofC}
\end{equation} We set
\begin{equation}
M_1:= \max\left\{ \udl{m}_1, 2 \alpha \right\} + \frac{T}{3} \| \ovl{U}_1\|_{\C^{0,1}_{t,x}}.
\label{eq:choiceofM}
\end{equation} With this choice of parameters, we apply Proposition \ref{proposition:lowvelocities} with $\tau= \frac{T}{3}$ and $M=M_1$, which gives $\ovl{U}_2$ and $\ovl{M}$. \par
This allows to set
\begin{equation}
\ovl{U}(t,x) := \left\{ \begin{array}{ll}
\ovl{U}_1(t,x), & (t,x) \in \left[0,\frac{T}{3}\right] \times \T^2, \\
\ovl{U}_2\left( t - \frac{T}{3},x \right), & (t,x) \in \left[ \frac{T}{3}, \frac{2T}{3} \right] \times \T^2, \\
\ovl{U}_1\left( t - \frac{2T}{3}, x \right), & (t,x) \in \left[ \frac{2T}{3}, T \right] \times \T^2.
\end{array} \right.
\label{eq:referencevectorfield}
\end{equation} By construction,
\begin{eqnarray}
&& \curl_x \ovl{U}(t,x) =0, \quad (t,x) \in [0,T] \times (\T^2 \setminus \omega), \label{eq:curlbar} \\
&& \supp \ovl{U} \subset (0,T) \times \omega, \label{eq:support} \\
&& \div_x \ovl{U}(t,x) = 0, \label{eq:incompressibility} \\
&& \int_{\T^2} \ovl{U}(t,x) \ud x = 0, \quad \forall t\in [0,T]. \label{eq:referencemoyennenulle}
\end{eqnarray} We set
\begin{equation}
\ovl{W}(t,x) := \curl_x \ovl{U}(t,x).
\label{eq:curl}
\end{equation} Let us consider the functions
\begin{equation*}
\mathcal{Z}_1(v) := -k_1 v_2e^{-\frac{|v|^2}{2}}, \quad \mathcal{Z}_2(v):= k_2 v_1 e^{-\frac{|v|^2}{2}}, \quad \quad \quad \forall v=(v_1,v_2) \in \R^2, 
\end{equation*} where $k_1,k_2 >0$ are normalisation constants. These functions satisfy that 
\begin{equation}
\mathcal{Z}_1, \mathcal{Z}_2 \in \S(\R^2),
\label{eq:referenceScwartz}
\end{equation} where $\S(\R^d)$ stands for the space of real-valued Schwartz functions in $\R^d$. Moreover, choosing $k_1, k_2$ adequately, we have
\begin{eqnarray}
&& \int_{\R^2} v_1 \mathcal{Z}_1(v) \ud v = \int_{\R^2} v_2 \mathcal{Z}_2(v) \ud v = 0, \label{eq:diagonalintegral} \\
&& \int_{\R^2} v_2 \mathcal{Z}_1(v) \ud v = - \int_{\R^2} v_1 \mathcal{Z}_2(v) \ud v = 1, \label{eq:antidiagonalintegral} \\
&& \int_{\R^2} \mathcal{Z}_1(v) \ud v = \int_{\R^2} \mathcal{Z}_2(v) \ud v = 0. \label{eq:Zintegralenulle} 
\end{eqnarray} We thus define, for any $(t,x,v) \in [0,T] \times \T^2 \times \R^2$,
\begin{equation}
\ovl{f}(t,x,v) := \mathcal{Z}_1(v) \d_{x_1} \ovl{W}(t,x) + \mathcal{Z}_2(v) \d_{x_2} \ovl{W}(t,x). 
\label{eq:referencefunction}
\end{equation} From (\ref{eq:support}) and (\ref{eq:curl}), we have
\begin{equation}
\ovl{f}|_{t=0} = 0, \quad \ovl{f}|_{t=T} = 0.
\label{eq:referenceinitialandfinal}
\end{equation} Thanks to (\ref{eq:Zintegralenulle}), 
\begin{equation}
\rho_{\ovl{f}}(t,x) = 0, \quad \forall (t,x) \in \Omega_T.
\label{eq:referencedensitenulle}
\end{equation}Furthermore, by construction and using (\ref{eq:diagonalintegral}) and (\ref{eq:antidiagonalintegral}) we find
\begin{equation*}
-\Delta_x \ovl{W} = \curl_x \int_{\R^2} v\ovl{f} \ud v .
\end{equation*} Moreover, thanks to (\ref{eq:referencefunction}), 
\begin{equation}
\int_{\T^2} j_{\ovl{f}}(t,x) \ud x = 0.
\label{eq:nulreferenceintergral}
\end{equation} Hence, for some $\ovl{p} \in \C^{\infty}([0,T] \times \T^2;\R)$, 
\begin{equation}
-\Delta_x \ovl{U}(t,x) = \nabla_x \ovl{p}(t,x) + \int_{\R^2} v\ovl{f} \ud v.  
\label{eq:referenceStokes}
\end{equation} Furthermore, from (\ref{eq:referencefunction}), (\ref{eq:curlbar}) and (\ref{eq:curl}), we deduce that 
\begin{eqnarray}
\d_t \ovl{f} + v\cdot \nabla_x \ovl{f} + \div_v \left[ (\ovl{U} - v )\ovl{f} \right] = 0, && \forall (t,x,v) \in [0,T] \times \left( \T^2\setminus \omega\right) \times \R^2, \nonumber \\
 \ovl{f}(t,x,v) = 0, && \forall (t,x,v) \in [0,T] \times (\T^2 \setminus \omega)  \times \R^2.  \label{eq:supportsolutionderefernce}
\end{eqnarray} To sum up, we have constructed a reference solution $(\ovl{f},\ovl{U},\ovl{p})$ of system (\ref{eq:VS}) with (\ref{eq:referenceinitialandfinal}) and such that the characteristics associated to $-v + \ovl{U}$ satisfy (\ref{eq:largevelocitycharacteristics}).


\section{Fixed point argument}

Let $\epsilon\in(0,\epsilon_0)$ be fixed, with $\epsilon_0$ to be chosen later on. We shall define an operator $\V_{\epsilon}$ acting on a domain $\S_{\epsilon} \subset \C^0([0,T]\times \T^2 \times \R^2)$ to be precised below. The goal of this section is to show that $\V_{\epsilon}$ has a fixed point.  

\subsection{Definition of the operator}
We describe the set $\S_{\epsilon}$. Let $\epsilon \in (0,\epsilon_0)$, to be precised later on, and $\gamma >2$. Then, set 
\begin{equation}
\delta_1:= \frac{\gamma}{2(\gamma+3)}, \quad \delta_2:= \frac{\gamma + 2}{\gamma+3}.
\label{eq:parametersdelta}
\end{equation} According to the notation of Section 1, we define
\begin{eqnarray}
&& \S_{\epsilon}:= \left\{ g \in \C^{0,\delta_2}_b (Q_T); \right. \nonumber \\
&& \quad \quad \quad  \textbf{(a)}\,\,  \left\| \int_{\R^2} (\ovl{f} - g) \ud v \right\|_{\C^{0,\delta_1}_b(\Omega_T)} \leq c_3 \epsilon, \nonumber \\
&& \quad \quad \quad  \textbf{(b)} \,\, \|(1+|v|)^{\gamma+2} (\ovl{f} - g) \|_{L^{\infty}(Q_T)} \nonumber \\
&& \quad \quad \quad  \quad \quad \quad  \quad \quad \leq c_1 \left( \|f_0\|_{\C^1(\T^2\times \R^2)} + \| (1+|v|)^{\gamma+2} f_0\|_{\C^0(\T^2 \times \R^2)} \right), \nonumber \\
&& \quad \quad \quad  \textbf{(c)} \,\,\, \|(\ovl{f} - g) \|_{\C^{0,\delta_2}_b(Q_T)} \nonumber \\
&& \quad \quad \quad  \quad \quad \quad  \quad \quad \leq c_2 \left( \|f_0\|_{\C^1(\T^2\times \R^2)} + \| (1+|v|)^{\gamma+2} f_0\|_{\C^0(\T^2 \times \R^2)} \right), \nonumber \\
&&  \quad \quad \quad \textbf{(d)} \, \int_{\T^2} \int_{\R^2} vg(t,x,v) \ud x \ud v = 0, \quad \forall t \in [0,T] \nonumber \\
&& \left. \quad \quad \quad \textbf{(e)} \, \int_{\T^2} \int_{\R^2} g(t,x,v) \ud x \ud v = \int_{\T^2} \int_{\R^2} f_0(x,v) \ud x \ud v, \quad \forall t \in [0,T]  \right\}, \nonumber 
\end{eqnarray} where $c_1,c_2,c_3$ are constants depending only on $T,\omega,\gamma, \delta_1$ and $\delta_2$ (see (\ref{eq:choiceofc1}), (\ref{eq:choiceofc2}) and (\ref{eq:choiceofc3}) for details).  We observe that, for $c_1,c_2,c_3$ large enough and $f_0\in \C^1(\T^2\times \R^2)$, with high moments in $v$, satisfying 
\begin{equation*}
\left\| \int_{\R^2} f_0(x,v) \ud v \right\|_{\C^{0,\delta_1}_b(\Omega_T)} \leq c_3\epsilon,
\end{equation*} we trivially have that $\ovl{f} + f_0 \in \S_{\epsilon}$. Thus, $\S_{\epsilon} \not = \emptyset$. \par
In order to describe the operator $\V_{\epsilon}$ we have to introduce some definitions. Let (see \cite[p. 342]{Glass})
\begin{eqnarray}
&& \gamma^{-}:= \left\{ (x,v) \in S(x_0,r_0) \times \R^2; \, |v|\geq \frac{1}{2}, \, \langle v, \nu(x) \rangle \leq -\frac{|v|}{10}   \right\}, \nonumber \\
&& \gamma^{2-}:= \left\{ (x,v) \in S(x_0,r_0) \times \R^2; \, |v|\geq 1, \, \langle v, \nu(x) \rangle \leq -\frac{|v|}{8}   \right\}, \nonumber \\
&& \gamma^{3-}:= \left\{ (x,v) \in S(x_0,r_0) \times \R^2; \, |v|\geq 2, \, \langle v, \nu(x) \rangle \leq -\frac{|v|}{5}   \right\}, \nonumber \\
&& \gamma^{+}:= \left\{ (x,v) \in S(x_0,r_0) \times \R^2; \, \langle v, \nu(x) \rangle \leq 0   \right\}, \nonumber 
\end{eqnarray} where $\nu(x)$ denotes the outward unit normal at $S(x_0,r_0)$ at $x$. It can be shown that 
\begin{equation}
\dist\left( [S(x_0,r_0) \times \R^2] \setminus \gamma^{2-}; \gamma^{3-}  \right) >0.
\label{eq:separationgamma}
\end{equation} Consequently, we may choose an absorption function $\A \in \C^{\infty} \cap W^{1,\infty} (S(x_0,r_0)\times \R^2;\R^+)$ such that 
\begin{eqnarray}
0 \leq \A(x,v)\leq 1, && \forall (x,v) \in S(x_0,r_0) \times \R^2, \label{eq:absortionfunction} \\
\A(x,v) = 1, && \forall (x,v) \in [S(x_0,r_0) \times \R^2 ]\setminus \gamma^{2-}, \nonumber  \\
\A(x,v) = 0, && \forall (x,v) \in \gamma^{3-}. \nonumber 
\end{eqnarray} We also choose a truncation function $\Y \in \C^{\infty}(\R^+;\R^+)$ such that 
\begin{eqnarray}
\Y(t) = 0, && \forall t\in \left[0,\frac{T}{48}\right] \cup \left[\frac{47T}{48}, T \right],\nonumber \\
\Y(t) = 1, && \forall t\in \left[\frac{T}{24},\frac{23T}{24}\right].\nonumber 
\end{eqnarray} We describe the operator $\V_{\epsilon}$ in three steps. 
\bigskip


\textbf{1. Stokes system.} Let $g\in \S_{\epsilon}$. We associate to $g$ the pair $(U^g(t),p^g(t))$, for every $t\in [0,T]$, solution of
\begin{equation}
\left\{ \begin{array}{ll}
-\Delta_x U^g(t) + \nabla_x p^g(t) = j_g(t)  , & x \in \T^2, \\
\div_x U^g(t,x) = 0, & x \in \T^2, \\
\int_{\T^2} U^g(t,x) \ud x = 0, & t\in [0,T],
\end{array} \right.
\label{eq:Stokesforg}
\end{equation} where
\begin{equation*}
j_g(t,x):= \int_{\R^2} v g(t,x,v) \ud v. 
\end{equation*} We shall prove that this association is well defined, thanks to point (d) and the following result. 
\begin{lemma}
Let $\epsilon >0$. Then, there exists a constant $K_1=K_1(\gamma)>0$ such that, for every $g\in \S_{\epsilon}$ and every $t\in[0,T]$,
\begin{equation*}
\| j_{g}(t) \|_{L^2(\T^2)^2} \leq K_1 \sqrt{ 1 + c_1^2\epsilon^2}.
\end{equation*} 
\label{lemma:jgisL2}
\end{lemma}

\begin{proof}
We write, by the triangular inequality,
\begin{align}
\|j_g(t)\|_{L^2(\T^2)^2}^2 & = \int_{\T^2} \left| \int_{\R^2} v g(t,x,v) \ud v \right|^2 \ud x \nonumber \\
& = \int_{\T^2} \left| \int_{\R^2} v \left( g - \ovl{f} + \ovl{f} \right)(t,x,v) \ud v \right|^2 \ud x \nonumber \\
& \leq \int_{\T^2} \left( \int_{\R^2} |v||g-\ovl{f}| \ud v + \int_{\R^2} |v||\ovl{f}| \ud v \right)^2 \ud x \nonumber \\
& \leq 2 \int_{\T^2} \left( \int_{\R^2} |v||g-\ovl{f}| \ud v  \right)^2 \ud x + 2 \int_{\T^2} \left( \int_{\R^2} |v||\ovl{f}| \ud v  \right)^2 \ud x.  \label{eq:lemma411}
\end{align} Let us note that, from (\ref{eq:referencefunction}), (\ref{eq:referenceScwartz}) and the properties of Schwartz functions, we have that
\begin{equation}
\mathcal{I}_1:= \int_{\T^2} \left( \int_{\R^2} |v||\ovl{f}| \ud v  \right)^2 \ud x < \infty,
\label{eq:lemma413}
\end{equation} is a positive constant, independent from $g$. \par 
We have to treat the first part of (\ref{eq:lemma411}). Indeed, by point (b),
\begin{align}
& \int_{\T^2} \left(  \int_{\R^2} |v||(g-\ovl{f})(t,x,v)| \ud v\right)^2 \ud x \nonumber \\
& \quad \quad  \leq  c_1^2 \left( \int_{\R^2} \frac{|v| \ud v}{(1+|v|)^{\gamma + 2}} \right)^2 (\|f_0\|_{\C^1} + \| (1+|v|)^{\gamma+2}f_0 \|_{L^{\infty}})^2, \nonumber \\
& \quad \quad \leq  \mathcal{I}_2 c_1^2 \epsilon^2, \label{eq:lemma412}
\end{align} where we have used (\ref{eq:smalldata}) and
\begin{equation*}
\mathcal{I}_2:= \left( \int_{\R^2} \frac{ |v| \ud v}{(1+|v|)^{\gamma +2}} \right)^2 < \infty.  
\end{equation*} Finally, putting together (\ref{eq:lemma412}), (\ref{eq:lemma413}) and (\ref{eq:lemma411}), we obtain the result by choosing
\begin{equation*}
K_1: = \sqrt{ 2 \max\left\{ \mathcal{I}_1, \mathcal{I}_2 \right\}}.
\end{equation*}
\end{proof}

We observe that the previous lemma shows that $j_g(t) \in L^2(\T^2)^2$, for every $t\in [0,T]$. Then, using point (d), we get from Proposition \ref{proposition:Stokesclassical} (see Appendix B for notation) the solution of (\ref{eq:Stokesforg}),
\begin{equation*}
(U^g(t),p^g(t)) \in (H^2_0(\T^2)^2 \cap \mathbb{V}) \times L^2(\T^2), \quad \forall t \in [0,T].
\end{equation*} Consequently, the association given by (\ref{eq:Stokesforg}) is well defined. \par 
Let us show some consequences that will be important in next sections. We will assume from now on that the choice of $\epsilon$ is made according to
\begin{equation}
\epsilon_0\leq \min \left\{ \frac{1}{c_1}, 1 \right\}.
\label{eq:choiceofepsilon0} 
\end{equation}  
\begin{lemma}
There exists a constant $K_2=K_2(T,\gamma)>0$ such that for any $g\in \S_{\epsilon}$ and $U^g$ given by (\ref{eq:Stokesforg}), we have 
\begin{equation}
\|U^g\|_{L^{\infty}(\Omega_T)} \leq K_2(T,\gamma).
\label{eq:majorationuniformechampdevitesse}
\end{equation}  Moreover,
\begin{equation} 
U^g \in \C^0([0,T];\C^1(\T^2;\R^2)).
\label{eq:regularvectorfield}
\end{equation}  
\label{lemma:champsuniformes}
\end{lemma}

\begin{proof}
Let $g \in \S_{\epsilon}$. Then, by (\ref{eq:regularityStokes}), the Sobolev embedding theorem and (\ref{eq:choiceofepsilon0}), 
\begin{eqnarray}
\|U^g\|_{L^{\infty}(\Omega_T)} &\leq & C_{S} \|U^g\|_{L^{\infty}_t (H^2_x)^2} \nonumber \\
& \leq & C_{S}C_1\|j_g\|_{L^{\infty}_t (L^2_x)^2} \leq C_SC_1 \sqrt{2} K_1,
\label{eq:vectorfielduniformly}
\end{eqnarray} which gives (\ref{eq:majorationuniformechampdevitesse}) for a constant $K_2:= \sqrt{2} C_SC_1 K_1>0$, independent from $g$. \par
Let us show (\ref{eq:regularvectorfield}). Indeed, by similar arguments as those in Lemma \ref{lemma:jgisL2}, we deduce that $j_g(t) \in L^{\infty}(\T^2)^2$, for every $t\in [0,T]$. Then, by interpolation between the $L^p$ spaces, we deduce that 
\begin{equation*}
j_g(t) \in L^p(\T^2)^2, \quad \forall t\in [0,T], \,  p\in[2,\infty]. 
\end{equation*} Consequently, using Proposition \ref{proposition:StokesLp} with source term $j_g(t) \in L^p(\T^2)^2$, we deduce that 
\begin{equation*}
U^g(t) \in W^{2,p}(\T^2)^2, \quad  \forall t\in [0,T], \, p\in [2,\infty). 
\end{equation*} Finally, choosing $2<p<\infty$, the Sobolev embedding theorem in this case (see \cite[Corollaire 9.1, p.52]{Lions}) implies that 
\begin{equation*}
U^g(t) \in \C^1(\T^2)^2, \quad  \forall t\in [0,T].  
\end{equation*} We have to show the continuity w.r.t. the time variable. Let $t,s \in [0,T]$. Then, $W(t,s):=U^g(t) - U^g(s)$, $p(t,s):=p^g(t) - p^g(s)$, satisfy
\begin{equation*}
\left\{ \begin{array}{ll} 
-\Delta_x W(t,s) + \nabla_x p(t,s) = j_g(t) - j_g(s), & x\in \T^2, \\
\div_x W(t,s) = 0, & x\in \T^2, \\
\int_{\T^2} W(t,s) \ud x =0. 
\end{array} \right.
\end{equation*} On the one hand, we have
\begin{align}
&\|j_g(t) - j_g(s)\|_{L^p(\T^2)^2}^p \label{eq:jst} \\
&\quad \leq C(p)\sup_{x\in \T^2} \left| \int_{\R^2} v(g(t,x,v) - g(s,x,v))\ud v \right| \nonumber \\
& \quad \leq C(p) \sup_{x\in \T^2} \|g(t,x) - g(s,x)\|_{L^{\infty}_v}^{1-\frac{\gamma +1}{\gamma +2}} \left( \int_{\R^2} |v||g(t,x,v) - g(s,x,v)|^{\frac{\gamma + 1}{\gamma + 2}} \ud v\right). \nonumber
\end{align} 
Furthermore, using point (b), we have, for any $x\in \T^2$,
\begin{align*}
&\int_{\R^2} |v||g(t,x,v) - g(s,x,v)|^{\frac{\gamma + 1}{\gamma + 2}} \ud v \\
& \quad \leq \int_{\R^2} |v|  |g(t,x,v) - \ovl{f}(t,x,v)|^{\frac{\gamma + 1}{\gamma + 2}} \ud v + \int_{\R^2} |v| |\ovl{f} (t,x,v)-\ovl{f}(s,x,v)|^{\frac{\gamma + 1}{\gamma + 2}} \ud v \\
& \quad \quad \quad \quad \quad \quad \quad \quad + \int_{\R^2} |v| | \ovl{f}(s,x,v) - g(s,x,v)|^{\frac{\gamma + 1}{\gamma + 2}}  \ud v \\
& \leq  2 \left[ c_1^{\frac{\gamma + 1}{\gamma + 2}} \left( \|f_0\|_{\C^1(\T^2\times \R^2)} +  \|(1+|v|)^{\gamma + 2}f_0 \|_{L^{\infty}(\T^2\times \R^2)}\right)^{\frac{\gamma + 1}{\gamma + 2}} \right] \int_{\R^2} \frac{|v| \ud v }{(1+|v|)^{\gamma + 1}}\nonumber \\
&  \quad \quad \quad \quad \quad \quad  \quad \quad + C(\ovl{f},\gamma)  \\
& \leq C(c_1,f_0,\gamma, \ovl{f}),
\end{align*} thanks to (\ref{eq:referencefunction}) and (\ref{eq:referenceScwartz}). This yields, from (\ref{eq:jst}) and using the fact that $g\in \C^{0,\delta_2}(Q_T)$,  
\begin{align*}
&\|j_g(t) - j_g(s)\|^p_{L^p(\T^2)^2} \\
& \quad \quad \leq C(p,c_1,f_0,\gamma, \ovl{f}) \sup_{x\in \T^2} \|g(t,x) - g(s,x)\|_{L^{\infty}_v}^{1-\frac{\gamma +1}{\gamma +2}} \\
& \quad \quad \leq C(p,c_1,f_0,\gamma, \ovl{f}) \|g\|_{\C^{0,\delta_2}(Q_T)}^{1 - \frac{\gamma + 1}{\gamma +2}}|t-s|^{\frac{1}{\gamma + 3}}.
\end{align*} Then, combining this with (\ref{eq:regularityStokes}), we deduce
\begin{equation*}
 \| W(t,s) \|_{W^{2,p}(\T^2)^2} \leq C |t-s|^{\frac{1}{p(\gamma + 3)}}, \quad \forall t,s \in [0,T].
\end{equation*} This implies, again by the Sobolev embedding,
\begin{equation*}
\lim_{|t-s| \rightarrow 0} \| U^g(t) - U^g(s) \|_{\C^1(\T^2)^2} = 0,
\end{equation*} which gives (\ref{eq:regularvectorfield}).

\end{proof}

We also deduce the following property for the backwards characteristics associated to $-v + U^g$.

\begin{lemma}
Let $g \in \S_{\epsilon}$ and let $(X^g,V^g)$ be the characteristics associated to  the field $-v + U^g$, according to (\ref{eq:Stokesforg}). Then, there exists a constant $K_3=K_3(T,\gamma)>0$, independent of $g$, such that
\begin{equation}
\left|e^t|v| - |V^g(0,t,x,v)| \right| \leq  K_3,
\label{eq:pointb}
\end{equation} for any $(t,x,v) \in [0,T] \times \T^2 \times \R^2$. 
\label{lemma:b}
\end{lemma} 

\begin{proof}
By (\ref{eq:characteristics}), we have
\begin{eqnarray}
\left|e^t|v| - |V^g(0,t,x,v)| \right| &\leq& \left| V^g(0,t,x,v) - e^tv   \right| \nonumber \\
&=& \left| \int_0^t e^s U^g(s, X^g(0,s,x,v)) \ud s \right| \nonumber \\
& \leq &  C(T) \| U^g \|_{L^{\infty}_{t,x}} \nonumber \\
& \leq &  C(T) K_2(T,\gamma), \nonumber
\end{eqnarray} using (\ref{eq:majorationuniformechampdevitesse}). This allows to conclude, choosing $K_3:=C(T) K_2(T,\gamma)$.
\end{proof}

\bigskip


\textbf{2. Absorption.} To give a sense to the procedure of absorption we need first the following result, which asserts that the number of times the characteristics associated to the Stokes velocity field of the previous part meet $\gamma-$ is finite.

\begin{lemma}
Let $g\in \S_{\epsilon}$ and let $U^g$ be given by (\ref{eq:Stokesforg}) accordingly. Let $(X^g,V^g)$ be the characteristics associated to the field $-v + U^g$. Then, for any $(x,v) \in \T^2 \times \R^2$, there exists $n(x,v) \in \N$ such that there exist $0 <t_1< \dots < t_{n(x,v)} <T$ such that
\begin{eqnarray}
&& \,\,\,  \left\{ (X^g,V^g)(t,0,x,v); \, t\in[0,T]   \right\} \cap \gamma^- = \left\{ t_i \right\}_{i=1}^{n(x,v)}, \label{eq:finitetimesofimpact}  \\
&& \, \,\, \exists s >0 \textrm{ s.t. } ( t_i - s, t_i +s) \cap (t_j - s, t_j + s) = \emptyset, \, \forall i,j =1, \dots, n(x,v), \label{eq:timesareisolated}
\end{eqnarray} with the convention that $n(x,v) = 0 $ and $\left\{ t_i \right\}_{i=1}^{n(x,v)} = \emptyset$ if $\left\{ (X^g,V^g)  \right\} \cap \gamma^- = \emptyset$.
\label{lemma:finitenumberofimpacts}
\end{lemma} For more details on this result, see \cite[p.348]{Glass} and \cite[p.5468]{GDHK1}. In the friction case, this holds true without further modification,  thanks to Lemma \ref{lemma:champsuniformes} and Lemma \ref{lemma:b}. \par

The previous lemma allows to define the following quantities. Let $f_0 \in \C^1(\T^2 \times \R^2)$ and let $(x,v) \in \T^2\times \R^2$. Then, for every $t_i, $ with $ i = 1,\cdots,n(x,v)$, we have $(\tilde{x},\tilde{v})=(X^g,V^g)(t_i,0,x,v) \in \gamma^-$. Moreover, let 
\begin{eqnarray}
&& f(t^-,\tilde{x},\tilde{v}) = \lim_{t\rightarrow t_i^-} f_0((X^g,V^g)(0,t,x,v)), \label{eq:fminus} \\
&& f(t^+,\tilde{x},\tilde{v}) = \lim_{t\rightarrow t_i^+} f_0((X^g,V^g)(0,t,x,v)). \label{eq:fplus}
\end{eqnarray}

We define $f:=\tilde{\V}_{\epsilon}[g]$ to be the solution of 
\begin{equation}
\left\{ \begin{array}{ll}
\d_t f + v\cdot \nabla_x f + U^g \cdot \nabla_v f - \div_v(vf) = 0, & (t,x,v) \in [0,T]\times [\T^2 \times \R^2]\setminus \gamma^{2-} \\
f(0,x,v) = f_0(x,v), &  (x,v) \in \T^2 \times \R^2, \\
f(t^+,x,v) = (1 - \Y(t))f(t^-,x,v) + \Y(t)\A(x,v) f(t^-,x,v), & (t,x,v) \in [0,T] \times \gamma^-. 
\end{array} \right.
\label{eq:absortion}
\end{equation} Let us explain how the absorption procedure works. From (\ref{eq:regularvectorfield}), the characteristics associated to the field $-v + U^g$ are regular. Thus, outside $\omega$, the system above defines a function $\tilde{\V}_{\epsilon}[g]$ of class $\C^1$. Moreover, the exact value of $\tilde{\V}_{\epsilon}[g]$ is given by these characteristics through (\ref{eq:explicitsolution}) and (\ref{eq:characteristics}). When the characteristics $(X^g, V^g)$ meet $\gamma^- $ at time $t$, $f(t^+,\cdot,\cdot)$ is fixed according to the last equation in (\ref{eq:absortion}). We can see the function $\Y (t) \A(x,v)$ as an opacity factor depending on time and on the incidence of the characteristics on $S(x_0,r_0)$. Indeed, $f(t^+,\cdot,\cdot)$ can take values varying from $f(t^-,\cdot,\cdot)$, in the case of no absorption, to $0$, according to the angle of incidence, the modulus of the velocity and time.
\par 

\bigskip


\textbf{3. Extension.} The function $\tilde{\V}_{\epsilon}[g]$ is not necessarily continuous around $[0,T] \times \gamma^{-} \subset [0,T] \times B(x_0, 2r_0)$. To avoid this problem we shall use some extension operators preserving regularity. \par 
Let us first consider a linear extension operator 
\begin{equation*} 
\ovl{\pi}:\C^0_b(\T^2\setminus B(x_0, 2r_0)) \rightarrow \C^0_b(\T^2),
\end{equation*} such that for any $\sigma \in (0,1)$, a $\C^{0,\sigma}_b$ function is mapped onto a $\C^{0,\sigma}_b$ function. This allows to define another linear extension operator by
\begin{equation*}
\begin{array}{cccc}
\tilde{\pi}: & \C^0_b([0,T] \times [\T^2 \setminus B(x_0,2r_0) ] \times \R^2 ) & \rightarrow & \C^0_b([0,T]\times \T^2 \times \R^2) \\
             & f & \mapsto & \tilde{\pi} f(t,x,v) = \ovl{\pi}\left[ f(t,\cdot, v)\right](x).
\end{array}
\end{equation*} We modify $\tilde{\pi}$ in the following way. Let $\mu_1,\mu_2 \in \C^{\infty}(\T^2 \times \R^2)$ such that
\begin{align*}
\int_{\T^2} \int_{\R^2} \mu_1 \ud x \ud v = 0,  & \quad \int_{\T^2} \int_{\R^2} v\mu_1 \ud x \ud v = 1, \\
\int_{\T^2} \int_{\R^2} \mu_2 \ud x \ud v = 1,  & \quad \int_{\T^2} \int_{\R^2} v\mu_2 \ud x \ud v = 0. 
\end{align*} Then, set
\begin{equation*}
\pi(f) := \tilde{\pi}(f) -\int_{\T^2} \int_{\R^2} v \tilde{\pi}(f)\ud x \ud v \mu_1 + \left[ \int_{\T^2}\int_{\R^2} \left( f_0 - \tilde{\pi}(f) \right) \ud x \ud v \right]\mu_2.
\end{equation*} Thus, $\pi$ is an affine extension satisfying the following properties: for every $f \in \C^0_b([0,T] \times (\T^2 \setminus B(x_0,2r_0)) \times \R^2))$, we have 
\begin{align}
& \int_{\T^2} \int_{\R^2} \pi{f}(t,x,v) \ud x \ud v = \int_{\T^2} \int_{\R^2} f_0(x,v) \ud x \ud v, \quad \forall t\in [0,T], \label{eq:totalintegralconservedbypi} \\
& \int_{\T^2} \int_{\R^2} v\pi{f}(t,x,v) \ud x \ud v = 0, \quad \forall t\in [0,T], \label{eq:momentdevitessebypi} \\
&  \exists C_{\pi}>0 \textrm{ such that } \label{eq:extensioninfty} \\ 
&\|(1+|v|)^{\gamma+2} \pi(f) \|_{L^{\infty}(Q_T)} \leq C_{\pi}\|(1+|v|)^{\gamma +2} f \|_{L^{\infty}([0,T]\times (\T^2\setminus \omega) \times \R^2)}, \nonumber  \\
& \forall \sigma \in (0,1), \, \exists C_{\pi,\sigma}>0 \textrm{ such that } \label{eq:extensionHolder} \\ 
&  \|\pi(f) \|_{\C_b^{0,\sigma}(Q_T)} \leq C_{\pi,\sigma}\|(1+|v|)^{\gamma +2} f \|_{\C_b^{0,\sigma}([0,T]\times (\T^2\setminus \omega) \times \R^2)}, \nonumber 
\end{align} We introduce another truncation in time. Let $\tilde{\Y} \in \C^{\infty}(\R^+;[0,1])$ such that
\begin{equation}
\begin{array}{ll}
\tilde{\Y}(t) =0, & t\in \left[0,\frac{T}{100}\right], \\
\tilde{\Y}(t) =1, & t\in \left[\frac{T}{48},T\right].
\end{array} 
\end{equation} Finally, we set
\begin{equation}
\begin{array}{cccc}
\Pi: & \C^0_b([0,T]\times (\T^2 \setminus B(x_0,2r_0)) \times \R^2) & \rightarrow & \C^0_b([0,T]\times \T^2 \times \R^2), \\
& f & \mapsto & \Pi f = (1 - \tilde{\Y}(t))f + \tilde{\Y}(t) \pi f.
\end{array}
\label{eq:definitionofpi}
\end{equation} This allows to define the fixed point operator by
\begin{equation}
\V_{\epsilon}[g]:= \ovl{f} + \Pi\left( \tilde{\V}_{\epsilon}[f]_{|([0,T]\times (\T^2 \setminus B(x_0,2r_0)) \times \R^2)\cup([0,\frac{T}{48}] \times \T^2 \times \R^2)} \right), 
\end{equation} for every $(t,x,v) \in [0,T] \times \T^2 \times \R^2$.

\subsection{Existence of a fixed point}
We shall apply the Leray-Schauder fixed point theorem (see \cite[Theorem 1.11, p. 279]{Gilbarg}). To do this, we have to verify that 
\begin{enumerate}
\item The set $\S_{\epsilon}$ is convex and compact in $\C^0_b(Q_T)$,
\item $\V_{\epsilon}: \S_{\epsilon} \subset \C^0_b(Q_T) \rightarrow \C^0_b(Q_T) $ is continuous,
\item $\V_{\epsilon}(\S_{\epsilon}) \subset \S_{\epsilon}$.
\end{enumerate} 

The first point is straightforward, since the convexity of $\S_{\epsilon}$ is clear and the compactness is a consequence of Ascoli's theorem (see, for instance, \cite[Theorem 11.28, p. 245]{Rudin}). The second point is similar to \cite[Section 3.3]{Glass} and holds without further modification, thanks to Lemma \ref{lemma:finitenumberofimpacts}, Lemma \ref{lemma:b} and Lemma \ref{lemma:Gronwall}. \par

We need to show that point (3) holds. Let $g \in \S_{\epsilon}$. We have to prove that $\V_{\epsilon}[g] \in \S_{\epsilon}$, i.e, points (a)--(c), since points (d) and (e) follow by the construction of $\V_{\epsilon}$, using (\ref{eq:momentdevitessebypi}) and (\ref{eq:totalintegralconservedbypi}). \par 

\subsubsection{Proof of point (b)} 
By construction of $\V_{\epsilon}$, we have 
\begin{align}
& \|(1+|v|)^{\gamma+2}\left( \V_{\epsilon}[f] - \ovl{f} \right) \|_{L^{\infty}(Q_T)} \nonumber \\
& \quad \quad = \left\|(1+|v|)^{\gamma+2} \Pi\left( \tilde{\V_{\epsilon}}[g]_{ |\left([0,T]\times (\T^2 \setminus B(x_0,2r_0)) \times \R^2 \cup [0,\frac{T}{48}] \times \T^2 \times \R^2 \right)}  \right)  \right\|_{L^{\infty}(Q_T)} \label{eq:majorationb} \\
& \quad \quad   \leq C_{\pi} \left\| (1+|v|)^{\gamma+2} \tilde{\V_{\epsilon}}[g]  \right\|_{L^{\infty}(Q_T)}, \nonumber 
\end{align} where we have used (\ref{eq:extensioninfty}). Moreover, by (\ref{eq:absortion}) and (\ref{eq:absortionfunction}), 
\begin{equation*}
|f(t^+,x,v)| \leq |f(t^-,x,v)|,
\end{equation*} which implies, through (\ref{eq:explicitsolution}),
\begin{equation*}
|\tilde{\V}_{\epsilon}[g](t,x,v)| \leq \left| e^{2t} f_0 \left( (X^g,V^g)(0,t,x,v) \right)  \right|.
\end{equation*} On the other hand,
\begin{align}
& \left|f_0\left( (X^g,V^g)(0,t,x,v) \right)\right| \nonumber \\
& =  \left( \frac{1+|V^g(0,t,x,v)|}{1+|V^g(0,t,x,v)|}  \right)^{\gamma+2} \left| f_0\left( (X^g,V^g)(0,t,x,v) \right)   \right| \nonumber \\
&\leq  \frac{\|(1+|v|)^{\gamma +2} f_0 \|_{L^{\infty}(Q_T)}}{ \left(  1 + |V^g(0,t,x,v)|  \right)^{\gamma+2}} \label{eq:majorationVg} \\
& = \frac{\|(1+|v|)^{\gamma +2} f_0 \|_{L^{\infty}(Q_T)}}{\left( 1 + \left[e^t|v| - (e^t|v| - |V^g(0,t,x,v)|) \right]  \right)^{\gamma+2}}  \nonumber \\
& \leq  \frac{ \left( 1 + \left| e^t|v| - |V^g(0,t,x,v)| \right| \right)^{\gamma+2} \|(1+|v|)^{\gamma +2} f_0 \|_{L^{\infty}(Q_T)} }{(1 + e^t|v|)^{\gamma +2}} \nonumber \\
& \leq  \frac{(1 + K_3(T,\gamma))^{\gamma + 2}  \|(1+|v|)^{\gamma +2} f_0 \|_{L^{\infty}(Q_T)}}{(1 + e^t|v|)^{\gamma+2}}, \nonumber
\end{align} where we have used (\ref{eq:pointb}) and the inequality (see \cite[Eq. (3.33), p. 347]{Glass}. 
\begin{equation}
\frac{1}{1+ |x-x'|} \leq \frac{1+|x'|}{1+|x|}, \quad  \forall x,x'\in \R^2.
\label{eq:inegalitequoitient}
\end{equation} Furthermore, since 
\begin{equation*}
(1+|v|)^{\gamma +2} |\tilde{\V_{\epsilon}}[g](t,x,v)| \leq (1 + e^t|v|)^{\gamma+2} |\tilde{\V_{\epsilon}}[g](t,x,v)|, 
\end{equation*} for every $(t,x,v)\in [0,T] \times \T^2 \times \R^2$, we have
\begin{align}
& \|(1+|v|)^{\gamma+2} \tilde{\V}_{\epsilon}[g] \|_{L^{\infty}(Q_T)} \label{eq:vtildepointb} \\ 
& \quad \quad \quad \quad \quad \leq e^{2T}\left( 1 +K_3(T,\gamma)\right)^{\gamma+2} \left(\|f_0\|_{\C^1} + \|(1+|v|)^{\gamma+2} f_0\|_{L^{\infty}} \right). \nonumber
\end{align} This gives that $\V_{\epsilon}[g]$ satisfies point (b), thanks to (\ref{eq:majorationb}) and choosing
\begin{equation}
c_1 \geq C_{\pi} e^{2T}\left( 1 + K_3(T,\gamma)\right)^{\gamma+2}.
\label{eq:choiceofc1}
\end{equation}


\subsubsection{Proof of point (c)} We need the following technical result, which can be adapted from \cite[Lemma 2, p. 347]{Glass}, thanks to Lemma \ref{lemma:Gronwall} and (\ref{eq:regularvectorfield}).

\begin{lemma}
For any $g\in \S_{\epsilon}$, one has $\tilde{\V}_{\epsilon}[g] \in \C^1(Q_T \setminus \Sigma_T)$, with $\Sigma_T := [0,T] \times \gamma^-$. Moreover, there exists a constant $K_4=K_4(\gamma,\omega)>0$ such that
\begin{equation*}
\frac{\left| \tilde{\V}_{\epsilon}[g](t,x,v) - \tilde{\V}_{\epsilon}[g](t',x',v')  \right|}{(1+|v|)|(t,x,v) - (t',x',v')|} \leq K_4(\|f_0\|_{\C^1(\T^2\times \R^2)} + \| (1+|v|)^{\gamma+2} f_0 \|_{L^{\infty}(Q_T)}),
\end{equation*} for any $(t,x,v),(t',x',v') \in [0,T]\times (\T^2 \setminus \omega) \times \R^2$ with $|v-v'|<1$. Furthermore, if $f_0$ satisfies (\ref{eq:conditionforuniqueness}), we also have
\begin{eqnarray}
&&\|(1+|v|)^{\gamma+1}\nabla_{x,v}\tilde{\V}_{\epsilon}[g]\|_{L^{\infty}} \nonumber \\
&& \quad \quad \quad \quad  \leq K_5 \left( \|(1+|v|)^{\gamma+1}\nabla_{x,v}f_0 \|_{L^{\infty}(Q_T)} + \| (1+|v|)^{\gamma + 2} f_0 \|_{L^{\infty}} \right), \label{eq:estimeeunicite}
\end{eqnarray} for some constant $K_5=K_5(\kappa, g)>0$.
\label{lemma:pointc}
\end{lemma} 

Let $\delta_2$ be given by (\ref{eq:parametersdelta}). Again, by construction of $\V_{\epsilon}$ and (\ref{eq:extensionHolder}), we deduce
\begin{equation}
\| \V_{\epsilon}[g] - \ovl{f} \|_{\C_b^{0,\delta_2}(Q_T)} \leq C_{\pi,\delta_2} \| \tilde{\V}_{\epsilon}[g] \|_{\C_b^{0,\delta_2}([0,T] \times (\T^2 \setminus B(x_0,2r_0)) \times \R^2)}.
\label{eq:majorationc}
\end{equation} Then, interpolating (\ref{eq:vtildepointb}) and Lemma \ref{lemma:pointc}, we have
\begin{eqnarray}
&& \frac{|\tilde{\V}_{\epsilon}[g](t,x,v) - \tilde{\V}_{\epsilon}[g](t',x',v')|}{|(t,x,v) - (t',x',v')|^{\delta_2}} \nonumber \\
&& = \left( \frac{|\tilde{\V}_{\epsilon}[g](t,x,v) - \tilde{\V}_{\epsilon}[g](t',x',v')|}{(1+|v|)|(t,x,v) - (t',x',v')|}  \right)^{\frac{\gamma +2}{\gamma +3}} \nonumber \\
&& \quad \quad \quad \quad \quad \quad \quad \quad \times \left( (1 + |v|)^{\gamma+2} |\tilde{\V}_{\epsilon}[g](t,x,v) - \tilde{\V}_{\epsilon}[g](t',x',v')| \right)^{1 - \frac{\gamma+2}{\gamma+3}} \nonumber \\
&& \leq K_4^{\frac{\gamma +2}{\gamma +3}}K_6^{1-\frac{\gamma +2}{\gamma +3}} \left( \|f_0\|_{\C^1(\T^2\times \R^2)} + \| (1+|v|)^{\gamma+2} f_0  \|_{L^{\infty}(Q_T)} \right), \nonumber 
\end{eqnarray} with 
\begin{equation*}
K_6 = 2e^{2T}\left( 1 + K_3(T,\gamma)\right)^{\gamma+2}.
\end{equation*} Whence, by (\ref{eq:majorationc}), this gives that $\tilde{\V}_{\epsilon}[g]$ satisfies point (c), choosing
\begin{equation}
c_2 \geq C_{\pi,\delta_2}K_4(\gamma,\omega)^{\frac{\gamma + 2}{\gamma + 3}} K_6(T,\gamma)^{1 - \frac{\gamma + 2}{\gamma + 3}}.
\label{eq:choiceofc2}
\end{equation}


\subsubsection{Proof of point (a)} We show first the $L^{\infty}$ estimate. Using (\ref{eq:referencedensitenulle}) and point (b), we find 
\begin{eqnarray}
\left\| \int_{\R^2} \left( \V_{\epsilon}[g] - \ovl{f} \right)  \ud v \right\|_{L^{\infty}(\Omega_T)} &=& \left\| \int_{\R^2} \left( \V_{\epsilon}[g](t,x,v)  \right)  \ud v \right\|_{L^{\infty}(\Omega_T)} \nonumber \\
& = &\sup_{(t,x) \in \Omega_T}  \int_{\R^2} \left|\V_{\epsilon}[g] (t,x,v) \right|  \ud v \nonumber \\
& \leq & K_7 (\|f_0\|_{\C^1} + \|(1+|v|)^{\gamma +2} f_0 \|_{L^{\infty}}),  \nonumber 
\end{eqnarray} with
\begin{equation*}
K_7 := c_1 \int_{\R^2} \frac{\ud v}{(1 + |v|)^{\gamma + 2}}.
\end{equation*} To show the H\"{o}lder estimate, we interpolate (\ref{eq:vtildepointb}) and (c). Indeed, if $\delta_1$ is given by (\ref{eq:parametersdelta}) and $\tilde{\gamma}:= 2 + \frac{\gamma }{2}$, we have 
\begin{eqnarray}
&& (1+|v|)^{\tilde{\gamma}}\frac{|\V_{\epsilon}[g](t,x,v) - \V_{\epsilon}[g](t',x',v)|}{|(t,x,v) - (t',x',v)|^{\delta_1}} \nonumber \\
&& = \left( (1 + |v|)^{\gamma+2} |\V_{\epsilon}[g](t,x,v) -\V_{\epsilon}[g](t',x',v)| \right)^{\frac{1}{2} + \frac{1}{\gamma +2}} \nonumber \\
&& \quad \quad \quad \quad \quad \quad \quad \times \left( \frac{|\V_{\epsilon}[g](t,x,v) - \V_{\epsilon}[g](t',x',v)|}{|(t,x,v) - (t',x',v)|^{\delta_2}} \right)^{\frac{1}{2} - \frac{1}{\gamma +2}} \nonumber \\
&& \leq c_1^{\frac{1}{2} + \frac{1}{\gamma +2 }} c_2^{\frac{1}{2} - \frac{1}{\gamma + 2}} \left( \|f_0\|_{\C^1(\T^2\times \R^2)} + \| (1+|v|)^{\gamma+2} f_0   \|_{L^{\infty}(Q_T)}  \right). \nonumber 
\end{eqnarray} Consequently, choosing
\begin{equation}
c_3 : = K_7 + c_1^{\frac{1}{2} + \frac{1}{\gamma +2 }} c_2^{\frac{1}{2} - \frac{1}{\gamma + 2}} \int_{\R^2} \frac{\ud v}{(1+|v|)^{\tilde{\gamma}}},
\label{eq:choiceofc3}
\end{equation} and thanks to (\ref{eq:smalldata}), we have that $\V_{\epsilon}[g]$ satisfies point (a). \par

\bigskip
Let us choose $\epsilon_0$ sufficiently small and satisfying (\ref{eq:choiceofepsilon0}). Then, the smallness assumption (\ref{eq:smalldata}) and the properties of $\V_{\epsilon} $ and $\Pi$ allow to conclude. Thus, if $\epsilon \leq \epsilon_0$, thanks to Leray-Schauder theorem, there exists $g \in \S_{\epsilon}$ such that $\V_{\epsilon}[g] = g$.


\subsection{Uniqueness}
The goal of this section is to show that the solution of (\ref{eq:VS}) obtained in the previous section is unique within a certain class. \par 
Indeed, let $\epsilon \leq \epsilon_0$ and $g=\V_{\epsilon}[g]$. Then. if $f_0\in \C^1(\T^2 \times \R^2)$ satisfies (\ref{eq:conditionforuniqueness}), Lemma \ref{lemma:pointc} gives (\ref{eq:estimeeunicite}). By the construction of $\ovl{f}$ and $\tilde{\V}_{\epsilon}$, and since $\Pi$ preserves regularity, we deduce that
\begin{align}
& g\in \C^1(Q_T), \label{eq:regulariteunicite} \\ 
& \exists \kappa'>0, \, \left( |g| + |\nabla_{x,v}g| \right)(t,x,v) \leq \frac{\kappa'}{(1 + |v|)^{\gamma+1}}, \quad \forall (t,x,v) \in Q_T, \label{eq:classeunicite} \\
& \int_{\T^2} j_{g}(t,x) \ud x = 0, \quad \forall t \in [0,T].   \label{eq:jgunicite}
\end{align} Next result, inspired from \cite[Section 8]{Ukai}, shows that the solution in this class is unique. 

\begin{proposition}
Let $f_0 \in \C^1(\T^2\times \R^2)$ satisfying (\ref{eq:conditionforuniqueness}) and let $G \in \C^0(Q_T)$. Then, the solution of system (\ref{eq:VS}) satisfying conditions (\ref{eq:regulariteunicite}), (\ref{eq:classeunicite}) and (\ref{eq:jgunicite}) is unique.
\label{proposition:uniqueness}
\end{proposition}

\begin{proof}
Let $f_1 = \tilde{\V}_{\epsilon}[f_1]$, for $\epsilon \leq \epsilon_0$. Let us suppose that $(f^2,U^2,p^2)$ is a solution of system (\ref{eq:VS}) with initial datum $f_0$ and control $G$ such that (\ref{eq:regulariteunicite}) and (\ref{eq:classeunicite}) are satisfied. \par 
Let $W:=U^1-U^2$, $g:=f^1 - f^2$, $\pi:= p^1 - p^2$. Then, $(g,W,\pi)$ satisfies
\begin{equation}
\left\{ \begin{array}{ll}
\partial_t g + v\cdot \nabla_x g + \div_v[(U^1-v)g] = -W\cdot \nabla_v f^2, & (t,x,v) \in Q_T, \\
-\Delta_x W +\nabla_x \pi = j_g, & (t,x)\in \Omega_T, \\
\div_x W =0, & (t,x) \in \Omega_T, \\
\int_{\T^2} W(t,x) \ud x = 0, & t\in [0,T], \\
g_{|t=0} = 0, & (x,v) \in \T^2 \times \R^2.
\end{array} \right.
\label{eq:systemeunicite}
\end{equation} Using Proposition \ref{proposition:Stokesclassical}, we get, for every $t\in [0,T]$, 
\begin{equation*}
\|W(t)\|_{(H_x^2)^2} \leq C \|j_g(t)\|_{(L^2_x)^2}.
\end{equation*} Moreover, the Sobolev embedding theorem gives, for every $t \in [0,T]$, 
\begin{equation}
\|W(t)\|_{(L^{\infty}_x)^2} \leq C' \|j_g(t)\|_{(L^2_x)^2}.
\label{eq:majorationchampdevitesses}
\end{equation} On the other hand, we observe that condition (\ref{eq:classeunicite}) gives 
\begin{align*}
& (1+|v|)|\nabla_{x,v}f^2(t,(X^1,V^1)(0,t,x,v))|  \\
& \quad \quad \quad \quad \quad \quad \leq \frac{\kappa' (1+|v|)}{(1+ |V^1(0,t,x,v)|)^{\gamma+1}}  \\
& \quad \quad \quad \quad \quad \quad \leq \frac{C(\kappa',\gamma)}{(1 + |v|)^{\gamma}},  
\end{align*} proceeding in the same fashion as in (\ref{eq:majorationVg}). As a result,
\begin{equation}
\sup_{(t,x) \in \Omega_T}\int_{\R^2} (1+|v|) \left|\nabla_v f^2 (t,(X^1,V^1)(0,t,x,v)\right| \ud v \leq \tilde{C}(\kappa',\gamma),
\label{eq:majorationgradf2}
\end{equation} for some constant $\tilde{C}(\kappa',\gamma)>0$. Consequently, from the Vlasov equation in (\ref{eq:systemeunicite}), by the method of characteristics, we have
\begin{eqnarray}
|g(t,x,v)| &\leq & e^{2T} \left|\int_0^t W(s,X^1(0,s,x,v))\cdot \nabla_v f^2(s,(X^1,V^1)(0,s,x,v) \ud s  \right| \nonumber \\
& \leq & e^{2T} \int_0^t \|W(s,\cdot)\|_{L^{\infty}_x} \left|\nabla_v f^2(s,(X^1,V^1)(0,s,x,v)\right| \ud s. \nonumber
\end{eqnarray} Thus,
\begin{equation*}
(1+|v|)|g(t,x,v)| \leq \int_0^t \|W(s,\cdot)\|_{L^{\infty}_x} (1+|v|)\left|\nabla_v f^2(s,(X^1,V^1)(0,s,x,v)\right| \ud s,
\end{equation*} which implies, thanks to (\ref{eq:majorationgradf2}) and (\ref{eq:majorationchampdevitesses}),
\begin{eqnarray}
\sup_{x \in \T^2 } |j_g(t,x)| &\leq & C\int_0^t \|W(s,\cdot)\|_{L^{\infty}_x} \ud s \nonumber \\
& \leq & C' \int_0^t \| j_g(s)\|_{(L^2_x)^2} \ud s \nonumber \\
& \leq & C'' \int_0^t \sup_{x\in \T^2 } \left|j_g(s)\right| \ud s, \quad \forall t\in [0,T], \nonumber
\end{eqnarray} which, by Gronwall's lemma entails, since $j_g(0) = 0$, that
\begin{equation*}
j_g(t,x) = 0, \quad \forall (t,x) \in \Omega_T.
\end{equation*} Moreover, using again system (\ref{eq:systemeunicite}) we deduce from this that the difference $W (t)= (U^1-U^2)(t)$ satisfies, for every $t \in [0,T]$,
\begin{equation*}
\left\{ \begin{array}{ll}
-\Delta_x W(t) +\nabla_x \pi(t) = 0, & x\in \T^2, \\
\div_x W(t) =0, & x \in \T^2, \\
\int_{\T^2} W(t)\ud x = 0, 
\end{array} \right.
\end{equation*} which, according to Proposition \ref{proposition:Stokesclassical} must imply that $U^1 = U^2$ in $\Omega_T$. In particular, the characteristics associated to $-v + U^1$ and to $-v + U^2$ coincide and thus, $f^1 = f^2$ in $Q_T$.
\end{proof}


\section{End of the proof}
In order to conclude the proof of Theorem \ref{thm:controllability} we have to show that if we choose $\epsilon \in (0, \epsilon_1)$ with $\epsilon_1$ small enough, then the fixed point $g$ found in the previous section satisfies (\ref{eq:finalstate}). To do this, we show that $\V_{\epsilon}[g]_{|t=T}=0$ in $(\T^2 \setminus \omega)\times \R^2$ (see the strategy of proof in Section 1.3.3). The key result is the following.

\begin{proposition}
There exists $\epsilon_1 >0$ such that the  characteristics $(X^g,V^g)$ associated to the field $  - v + U^g$ meet $\gamma^{3-}$ for some time in $\left[ \frac{T}{24}, \frac{23T}{24} \right]$. 
\label{proposition:gamma4}
\end{proposition}

\begin{proof}
Let us define 
\begin{equation*}
\gamma^{4-}:= \left\{ (x,v) \in S(x_0,r_0) \times \R^2; \quad |v| \geq \frac{5}{2}, \, \langle v, \nu(x) \rangle \leq - \frac{|v|}{4} \right\}.
\end{equation*} We proceed in two steps. In a first time, we show the result for the characteristics associated to $-v+\ovl{U}$, given by (\ref{eq:referencevectorfield}). In a second time, we show that, thanks to the first step, the result for $(X^g,V^g)$ follows by choosing $\epsilon_1>0$ small enough.

\begin{description}

\item[Step 1] Let us consider the characteristics $(\ovl{X},\ovl{V})$ associated to the field $-v + \ovl{U}$. We claim that
\begin{equation}
\exists \, \sigma \in \left[ \frac{T}{12}, \frac{3T}{12} \right] \cup \left[ \frac{9T}{12},\frac{11T}{12} \right] \textrm{ such that } \ovl{X}(\sigma,0,x,v) \in \gamma^{4-}.
\label{eq:claimend}
\end{equation} To show this claim, we need to prove the following 
\begin{equation}
\exists \, t \in \left[ \frac{T}{12}, \frac{3T}{12} \right] \cup \left[ \frac{9T}{12},\frac{11T}{12} \right] \textrm{ s.t } \ovl{X}(t,0,x,v) \in B\left( x_0, \frac{r_0}{4} \right), \, |\ovl{V}(t,0,x,v)|\geq \alpha,
\label{eq:claimend2}
\end{equation} where $\alpha$ is given by (\ref{eq:choiceofalpha}). Indeed, let $M_1$ be given by (\ref{eq:choiceofM}) and let us consider two cases.
\begin{description}

\item[Case 1] If $|\ovl{V}(\frac{T}{3},0,x,v)| \geq M_1$, then using (\ref{eq:characteristics}), one obtains that 
\begin{align*}
|v| & \geq  e^{\frac{T}{3}} \left|\ovl{V} (\frac{T}{3},0,x,v) \right| - e^{\frac{T}{3}}\left| \int_0^{\frac{T}{3}} e^{\tau-\frac{T}{3}} \ovl{U}_1 (\tau, \ovl{X}(\tau,0,x,v)) \ud \tau \right|  \\
& \geq  e^{\frac{T}{3}} M_1 - \frac{T}{3}e^{\frac{T}{3}}\| \ovl{U}_1 \|_{\C^{0,1}_{t,x}}. 
\end{align*} Then, by the choice (\ref{eq:choiceofM}), this implies
\begin{equation}
|v| \geq e^{\frac{T}{3}}\max\left\{ \udl{m}_1, 2\alpha \right\}.
\label{eq:minorationM}
\end{equation}   Hence, in particular, we get $|v| \geq \udl{m}_1$. This allows to apply Proposition \ref{proposition:highvelocities}, which gives that $\exists t \in \left[ \frac{T}{12}, \frac{3T}{12} \right]$ such that $\ovl{X}(t,0,x,v) \in B(x_0,\frac{r_0}{4})$. \par 
Morover, we deduce from (\ref{eq:minorationM}) that 
\begin{equation*}
|v| \geq 2e^{\frac{T}{3}}\alpha,
\end{equation*} which entails, thanks to (\ref{eq:largevelocitycharacteristics}), that $|\ovl{V}(t,0,x,v)| \geq \alpha$. Thus, (\ref{eq:claimend2}) is satisfied in this case. 

\item[Case 2] If $|\ovl{V}(\frac{T}{3},0,x,v)| < M_1$, then Proposition \ref{proposition:lowvelocities} implies that 
\begin{equation*}
\left|\ovl{V}(\frac{2T}{3},0,x,v) \right| \geq 1 + M_1.
\end{equation*} Proceeding as in the previous case, this  yields (\ref{eq:claimend2}) with some $t \in \left[ \frac{9T}{12}, \frac{11T}{12} \right]$.
\end{description} This shows (\ref{eq:claimend2}). \par 

Let us prove (\ref{eq:claimend}). We choose $s>0$ with 
\begin{equation*}
s < \log \left( 1 + \frac{9r_0}{ 4\alpha}  \right) < \frac{T}{200},
\end{equation*} thanks to (\ref{eq:choiceofalpha}) and (\ref{eq:choiceofC}), and such that 
\begin{eqnarray}
&& \ovl{X}(t,0,x,v) + (1 - e^s)\ovl{V}(t,0,x,v) \in S(x_0,2r_0), \label{eq:propertysss} \\
&& \langle \ovl{V}(t,0,x,v) , \nu(x) \rangle \leq -\frac{\sqrt{3}}{2}|\ovl{V}(t,0,x,v)|. \label{eq:propertiyvitesesss}  
\end{eqnarray} The last point follows from the fact that any straight line arising from $B(x_0, \frac{r_0}{4})$ cuts $S(x_0,r_0)$ forming an angle with $\nu(x)$ of value at most $\frac{\pi}{6}$. To see (\ref{eq:propertysss}), we observe that, choosing $s_0:=\log \left( 1 + \frac{9r_0}{4\alpha} \right)$ one has
\begin{eqnarray}
&& |\ovl{X}(t,0,x,v) + (1-e^{s_0})\ovl{V}(t,0,x,v) - x_0 | \nonumber \\
&& \quad \quad \quad \quad  \quad \quad \quad \geq (e^{s_0} -1)\alpha  - |\ovl{X}(t,0,x,v) - x_0 | \nonumber \\
&& \quad \quad \quad \quad  \quad \quad \quad \geq \frac{9r_0}{4} - \frac{r_0}{4} = 2 r_0. \nonumber 
\end{eqnarray} Hence, by the intermediate value theorem, there exists $s$ with $0<s \leq s_0$ such that (\ref{eq:propertysss}) holds.\par 
Moreover, we deduce from (\ref{eq:characteristics}) and (\ref{eq:claimend2}) that
\begin{eqnarray}
&& \left| \ovl{X}(t,0,x,v) + (1 - e^s)\ovl{V}(t,0,x,v)   - \ovl{X}(t-s,0,x,v)  \right| \nonumber \\
&& \quad \quad  = \left| \ovl{X}(t,0,x,v) + (1 - e^s)\ovl{V}(t,0,x,v)   - \ovl{X}(t-s,t,(\ovl{X},\ovl{V})(t,0,x,v))  \right| \nonumber \\
&& \quad \quad   = \left| \int_t^{t-s} \int_t^{\sigma} e^{\sigma - z} \ovl{U}_1(z,\ovl{X}(z,\ovl{X}(t,0,x,v))) \ud z \ud \sigma \right| \nonumber \\
&& \quad \quad   \leq Te^Ts\|\ovl{U}_1\|_{\C^0_{t,x}} \nonumber \\
&& \quad \quad   \leq C(T) \log \left( 1 + \frac{9r_0}{4\alpha} \right) \|\ovl{U}_1\|_{\C^0_{t,x}} \nonumber \\
&& \quad \quad   \leq C'(T) \frac{9r_0}{4\alpha} \|\ovl{U}_1\|_{\C^0_{t,x}} \leq \frac{C(T,r_0)}{\mathcal{C}_{r_0,T}}, \nonumber
\end{eqnarray} using (\ref{eq:choiceofalpha}) and the fact that $\log(1 + x)\leq x$ for $x$ small. We may choose $\mathcal{C}_{r_0,T}$ large enough, together with (\ref{eq:choiceofC}), so that
\begin{equation*}
\left| \ovl{X}(t,0,x,v) + (1 - e^s)\ovl{V}(t,0,x,v) - \ovl{X}(t-s,0,x,v)  \right| < \frac{r_0}{2},
\end{equation*} which allows to deduce, thanks to (\ref{eq:propertysss}), that 
\begin{equation*}
\ovl{X}(t-s,0,x,v) \not \in B(x_0,r_0).
\end{equation*} Hence, by the intermediate value theorem, there exists $\sigma \in \left[t-s,t  \right]$ such that 
\begin{equation*}
\ovl{X}(\sigma,0,x,v) \in S(x_0,r_0).
\end{equation*} Moreover, by (\ref{eq:characteristics}) and (\ref{eq:claimend2}), we have
\begin{align*}
|\ovl{V}(\sigma,0,x,v)| & = |\ovl{V}(\sigma,t,(\ovl{X},\ovl{V})(t,0,x,v)| \\
& \geq e^{t-\sigma}|\ovl{V}(t,0,x,v)| - Te^{t-\sigma} \|\ovl{U}_1 \|_{\C^{0,1}_{t,x}} \\
& = e^{t-\sigma}(\alpha - T\|\ovl{U}_1\|_{\C^{0,1}_{t,x}}).
\end{align*} Then, the choice of $\alpha$ in (\ref{eq:choiceofalpha}) yields
\begin{equation*}
|\ovl{V}(\sigma,0,x,v)| \geq \frac{5}{2}.
\end{equation*} Thus, (\ref{eq:claimend}) follows.

\item[Step 2]
Let us denote by $(X^g,V^g)$ the characteristics associated to $-v + U^g$. We have
\begin{equation}
\sup_{(t,x,v) \in Q_T}\left| ((X^g,U^g) - (\ovl{X},\ovl{V}))(t,0,x,v) \right|  \leq  C \|U^g - \ovl{U}\|_{\C^0_{t,x}}.  
\label{eq:approximation}
\end{equation} Observe that, thanks to (\ref{eq:Stokesforg}), (\ref{eq:referenceStokes}), (\ref{eq:referencemoyennenulle}) and (\ref{eq:nulreferenceintergral}), $(U^g - \ovl{U})(t)$ satisfy 
\begin{equation*}
\left\{  \begin{array}{ll}
-\Delta_x (U^g - \ovl{U})(t) + \nabla_x (p^g - \ovl{p})(t) = j_{g-\ovl{f}}(t), & x\in \T^2,\\
\div_x (U^g - \ovl{U})(t) = 0, & x\in \T^2, \\ 
\int_{\T^2} (U^g - \ovl{U})(t) \ud x =0, & t\in[0,T].
\end{array} \right.
\end{equation*} Using the Sobolev embedding theorem and Proposition \ref{proposition:Stokesclassical}, we deduce
\begin{eqnarray}
\|U^g - \ovl{U}\|_{(L^{\infty}_{t,x})^2} & \leq & C\|U^g - \ovl{U}\|_{L^{\infty}_t (H^2_x)^2} \nonumber \\
& \leq & C\| j_{g-\ovl{f}}\|_{L^{\infty}_t L^2(\T^2)^2} \leq  C\epsilon, \nonumber 
\end{eqnarray} using point (b). Hence, choosing $\epsilon_1$ small enough, from (\ref{eq:claimend}) and (\ref{eq:approximation}), the conclusion follows.

\end{description}

\end{proof}


\begin{proof}[Proof of Theorem \ref{thm:controllability}]
By Section 4.2, choosing $\epsilon \leq \min\left\{ \epsilon_0,\epsilon_1 \right\}$, where $\epsilon_0$ satisfies (\ref{eq:choiceofepsilon0}) and $\epsilon_1$ is given by Proposition \ref{proposition:gamma4}, there exists $g\in \S_{\epsilon}$ such that $\V_{\epsilon}[g] = g$ and such that Proposition \ref{proposition:gamma4} applies. \par
The fact that $g$ satisfies system (\ref{eq:VS}) follows from the construction of $\V_{\epsilon}$ and (\ref{eq:absortion}). Since $g\in\C^1([0,T]\times \T^2 \times \R^2)$, thanks to Lemma \ref{lemma:pointc}, (\ref{eq:supportsolutionderefernce}) and the fact that $\Pi$ preserves regularity, we deduce that
\begin{equation*}
\d_t g + v\cdot\nabla_v g - \div_v \left[(U^g - v)g   \right] = 1_{\omega}(x)G,
\end{equation*} for some $G\in \C^0([0,T]\times \T^2 \times \R^2)$. \par
To show (\ref{eq:finalstate}), we observe that from the construction of $\V_{\epsilon}$ and (\ref{eq:referenceinitialandfinal}), we have
\begin{equation}
\V_{\epsilon}[g](T,x,v):= \Pi\left( \tilde{\V}_{\epsilon}[f]_{|([0,T]\times (\T^2 \setminus B(x_0,2r_0)) \times \R^2)\cup([0,\frac{T}{48}] \times \T^2 \times \R^2)} \right)(T,x,v).
\end{equation} In particular, for any $(x,v) \in (\T^2 \setminus \omega) \times \R^2$, it comes from the definition of $\Pi$ that 
\begin{equation*}
\V_{\epsilon}[g](T,x,v) = \tilde{\V}_{\epsilon}[g](T,x,v).
\end{equation*} Moreover, by (\ref{eq:absortion}) and (\ref{eq:explicitsolution}),
\begin{equation*}
\tilde{\V}_{\epsilon}[g](T,x,v) = e^{2T} f_0((X^g,V^g)(0,T,x,v)).
\end{equation*} Hence, the absorption procedure described by (\ref{eq:absortion}) and (\ref{eq:absortionfunction}) and Proposition \ref{proposition:gamma4} allow to conclude that $\tilde{\V}_{\epsilon}[g](T,x,v) =0$ in $(\T^2\setminus \omega) \times \R^2$.

\end{proof}

\section{Conclusion and perspectives}
We have proved in Theorem \ref{thm:controllability} a local controllability result for the Vlasov equation coupled with the stationary Stokes system. Some possible extensions are possible. \par  

We could consider the Vlasov-Stokes system on a bounded domain with boundary, as in \cite{Hamdache}, with Dirichlet boundary conditions for the vector field and specular boundary conditions for the distribution function. In this case, with an internal control, the construction of a reference trajectory given in Section 3 is no longer effective, because of the specular reflection on the boundary of the characteristic flow. In particular, the distinction between good and bad directions should be refined. A geometric control condition could be very useful in this context. The boundary control problem may necessitate a very technical approach.  \par 

Other fluid-kinetic models could possibly be studied with similar techniques, such as the Vlasov-Navier-Stokes system. This is a matter of current work.

\begin{appendix}

\section{Auxiliary results on harmonic approximation}
We gather some results needed for the construction of the reference trajectory in Section 3. \par
As it has been done in by O. Glass in \cite{Glass}, the treatment of large velocities relies on a result on harmonic approximation due to T. Bagby and P. Blanchet (see \cite{BB}).

\begin{proposition}[\cite{BB}]
Let $\mathscr{F}$ be a closed subset of an orientable compact Riemannian manifold $\Omega$, and $\mathscr{U}$ an open subset of $\Omega \setminus \mathscr{F}$. Suppose that $\mathscr{U}$ meets every connected component of $\Omega \setminus \mathscr{F}$. For $f$ harmonic in a neighborhood of $\mathscr{F}$ and $\epsilon >0$, there is a Newtonian function $u$ on $\Omega$, whose poles lie in $\mathscr{U}$, and such that 
\begin{equation}
\sup_\mathscr{F} |u-f| < \epsilon.
\end{equation}
\label{proposition:BB}
\end{proposition}  

This result allows to show the following, which is a minor variation of \cite[Lemma A.1, p. 374]{Glass}. Let $\left\{e_1, \dots, e_N  \right\}$ be given by Definition \ref{definition:baddirections}.

\begin{proposition}
For any $i\in \left\{1,\dots,N \right\}$ and any $\epsilon >0$, there exists $\theta^i \in \C^{\infty}(\T^2;\R)$ such that 
\begin{eqnarray}
&& \Delta \theta^i(x) = 0, \quad x \in \T^2 \setminus B\left( x_0, \frac{r_0}{10} \right),\label{eq:largephiharmonic}  \\
&& \| \nabla \theta^i - e_i \|_{\C^1(\T^2 \setminus [B(x_0, r_0/10) + \R e_i])} \leq \epsilon. \label{eq:largephiapproximation}
\end{eqnarray}
\label{proposition:highvelocitiesapproximation}
\end{proposition}

For the treatment of low velocities, we need the following result, proved by O. Glass in \cite[Lemma 3, p. 356]{Glass}.

\begin{proposition}[O. Glass \cite{Glass}]
For any nonempty open set $\O \subset \T^2$, there exists $\theta \in \C^{\infty}(\T^2;\R)$ such that 
\begin{eqnarray}
\Delta \theta(x) = 0, && \forall x\in \T^2\setminus \O, \label{eq:thetaharmonic} \\
|\nabla \theta (x) | >0 && \forall x \in \T^2 \setminus \O. \label{eq:thetapositivemodulus}
\end{eqnarray}
\label{proposition:auxiliarylowvelocities}
\end{proposition}

\section{The Stokes system}
The well-posedness theory for the Stokes system is classical, especially in the case of $L^2$ with possibly Dirichlet boundary conditions. However, we shall need to precise an energy estimate used in Section 4 and and a regularity result in $L^p$ spaces. \par
Following \cite[Ch.2]{Temam}, we set the appropriate functional setting. We shall work with the usual Sobolev spaces  $W^{m,p}(\T^2)$, with $m\in \N$ and $1\leq p \leq \infty$. When $p=2$, we can write, thanks to the Fourier series,
\begin{align*}
&H^m(\T^2) = \left\{ f\in L^2; \, f=\sum_{n\in\Z^2} c_n e^{in\cdot x}, \, \ovl{c_n} = c_{-n}, \, \sum_{n\in\Z^2} \left( 1 + |n|\right)^{2m}|c_n|^2 < \infty \right\}, \\
& H^m_0(\T^2)  = \left\{ f\in H^m(\T^2); \, \int_{\T^2} f(x)\ud x =0 \right\},
\end{align*} which allows to equip these spaces, respectively, with the norms
\begin{equation*}
\|f \|_{H^m} := \left( \sum_{n\in\Z^2}(1 + |n|)^{2m}|c_n|^2  \right)^{\frac{1}{2}}, \quad \|f \|_{H_0^m} := \left( \sum_{n\in\Z^2}|n|^{2m}|c_n|^2  \right)^{\frac{1}{2}},
\end{equation*} with equivalence of norms in the case of $H^m_0$ as a subspace of $H^m$.\par 
In the case of vector fields, we shall use $(W^{m,p}(\T^2))^2$, with the product norm. Let us introduce, as usual,
\begin{equation*}
\mathbb{V}:= \left\{ F \in H^1(\T^2)^2; \, \div_x F = 0 \textrm{ in } \R^2 \right\},
\end{equation*} where the operator $\div_x$ is taken in the distributional sense. This setting allows to treat the system (see \cite[Section 2.2]{Temam}) 
\begin{equation}
\left\{ \begin{array}{ll}
-\Delta_x U + \nabla_x p = f, & x \in \T^2, \\
\div_x U = 0, & x\in \T^2, \\
\int_{\T^2} U \ud x =0.
\end{array} \right.
\label{eq:homogeneousStokes}
\end{equation}

\begin{proposition}[\cite{Temam}]
Let $f\in L^2(\T^2)^2$ such that $\int_{\T^2} f(x)\ud x =0$. Then, the Stokes system (\ref{eq:homogeneousStokes}) has a unique solution $(U,p) \in \left( H^2_0(\T^2)^2 \cap \mathbb{V} \right)\times H^1(\T^2)$. Moreover, there exists a constant $C_0>0$ such that 
\begin{equation}
\|U\|_{H^2} + \|p\|_{H^1} \leq C_0 \|f\|_{L^2}.
\label{eq:regularityStokes}
\end{equation}
\label{proposition:Stokesclassical}
\end{proposition}

The following regularity result in $L^p$ spaces is an adaptation of \cite[Theorem 3, p. 172]{Amrouche} to the case of the flat torus $\T^2$.
\begin{proposition}[\cite{Amrouche}]
Let $r\in (1,\infty)$ and $m\in \N$. For each $f \in W^{m,r}(\T^2)$ satisfying $\int_{\T^2}f(x)\ud x = 0$, the Stokes system (\ref{eq:homogeneousStokes}) has a unique solution $U \in W^{m+2,r}(\T^2)$, $p \in W^{m+1,r}(\T^2)$. Moreover, 
\begin{equation}
\| U \|_{W^{m+2,p}(\T^2)^2} + \|p\|_{W^{m+1,p}(\T^2)} \leq C_1 \| f\|_{W^{m,p}(\T^2)^2},
\label{eq:energyestimateLp}
\end{equation} for a constant $C_1>0$.
\label{proposition:StokesLp}
\end{proposition}

\end{appendix}

\addcontentsline{toc}{section}{Bibliography}       
\bibliographystyle{plain}                            

\begin{thebibliography}{}

\end{thebibliography}


\begin{thebibliography}{99}
\bibitem{Amrouche} C. Amrouche and V. Girault. On the existence and regularity of the Stokes problem in arbitrary dimension. \emph{Proc. Japan. Acad.} vol. 67, Ser. A. 1991.

\bibitem{BB} T. Bagby and P. Blanchet. Uniform harmonic approximation on Riemannian manifolds. \emph{J. Anal. Math}. vol. 62. pp. 47-76. 1994.

\bibitem{Bardos} C. Bardos, G. Lebeau and J. Rauch. Sharp sufficient conditions for the observation, control, and stabilization of waves from the boundary. \emph{SIAM J. Control Optim.} vol. 30:5, pp. 1024-1065. 1992.

\bibitem{Boudin} L. Boudin, C. Grandmont, A. Lorz and A. Moussa. Modelling and numerics for respiratory aerosols. To appear in \emph{Comm. in Comput. Physics}. 2015.


\bibitem{Coron} J.-M. Coron. Control and Nonlinearity. Mathematical Surveys and Monographs, vol. 136. American Mathematical Society, 2007.

\bibitem{Golse} L. Desvillettes, F. Golse and V. Ricci. The mean-field limit for solid particles in a Navier-Stokes flow. \emph{J. Statistical Physics.} vol. 131:5, pp. 941-967. 2008.

\bibitem{GJP} I. Gasser, P. E. Jabin and B. Perthame. Regularity and propagation of moments in some nonlinear Vlasov systems. \emph{Proc. Roy. Soc. Edinburgh Sect. A}, vol. 130, pp.1259-1273. 2000. 

\bibitem{Gilbarg} D. Gilbarg and N.S. Trudinger. Elliptic Partial Differential Equations of Second Order. \emph{Grunderlehren der mathematischen Wissenschaften. Springer}. vol. 224. 1998. 

\bibitem{Glass} O. Glass. On the controllability of the Vlasov-Poisson system. \emph{J. of Diff. Eq.} vol. 195, pp. 332-379. 2003.

\bibitem{GlassBourbaki} O. Glass. La m\'{e}thode du retour en contr\^{o}labilit\'{e} et ses applications en m\'{e}canique des fluides (d'apr\`{e}s J.-M. Coron et al.) \emph{S\'{e}minaire Bourbaki}. Novembre 2010.

\bibitem{GDHK1} O. Glass and D. Han-Kwan. On the controllability of the Vlasov-Poisson system in the presence of external force fields. \emph{J. of Diff. Eq.} vol. 252. pp. 5453-5491. 2012.

\bibitem{GDHK2} O. Glass and D. Han-Kwan. On the controllability of the relativistic Vlasov-Maxwell system. \emph{J. Math. Pures et Appl.} vol. 103. pp. 695-740. 2015.


\bibitem{Hamdache} K. Hamdache. Global existence and large time behaviour of solutions for the Vlasov-Stokes equations. \emph{Japan J. Industr. Appl. Math.} vol. 15. pp. 51-74. 1998.

\bibitem{Jabinfriction} P. E. Jabin. Macroscopic limit of Vlasov type equation with friction. \emph{Ann. IHP Anal. Non Lineaire.} vol.17, pp. 651-672. 2000.

\bibitem{JabinM3AS} P. E. Jabin. Various levels of models for aerosols. \emph{Math. Models Methods Appl. Sci.} vol. 12, pp. 903-919. 2002.

\bibitem{Jabinasymp} P. E. Jabin. Large time concentrations for solutions to kinetic equations with energy dissipation. \emph{Comm. Partial Differential Equations.} vol.25, pp. 541-557. 2000.

\bibitem{Jabin} P. E Jabin and B. Perthame. Notes on mathematical problems on the dynamics of dispersed particles interacting through a fluid. Modelling in applied sciences, a kinetic theory approach. \emph{Model. Simul. Sci. Eng. Tech.} Birkhauser. Boston. 2000.

\bibitem{Lions} J.-L. Lions and E. Magenes. Problemes aux limites non homogenes. \emph{Dunod}. Paris. 1968.

\bibitem{Rudin} W. Rudin. Real and Complex Analysis. \emph{McGraw-Hill.} 3rd ed. 1987.

\bibitem{Temam} R. Temam. Navier-Stokes equations and Nonlinear Functional Analysis. \emph{CBMS-NSF Reg. Conf in Appl. Math. SIAM}. 1983.

\bibitem{Ukai} S. Ukai and T. Okabe. On classical solutions in the large in time of two-dimensional Vlasov's equation. \emph{Osaka J. Math.} vol. 15:2. pp. 245-261. 1978.

\end{thebibliography}

\end{document}